\documentclass[12pt]{amsart}

\usepackage{amsfonts,amsmath,amsthm,amsxtra,amscd,amssymb}
\usepackage{amsmath}
\usepackage{amsthm}
\usepackage{amsxtra}
\usepackage{tikz}

\usepackage{amssymb,verbatim}
\usepackage{bm}
\usepackage{newlfont}
\usepackage{txfonts}
\usepackage{color}
\usepackage{mathptmx}
\usepackage{bm}
\usepackage{palatino}

\def\sfl{static feedback linearizable}
\def\dfl{dynamic feedback linearizable}
\def\stf{static feedback}
\def\df{dynamic feedback}

\theoremstyle{plain}
\newtheorem{thm}{Theorem}
\newtheorem{prop}{Proposition}
\newtheorem{lem}{Lemma}
\newtheorem{cor}{Corollary}

\newtheorem*{proc contact}{Procedure Contact}

\theoremstyle{definition}
\newtheorem{exmp}{Example}[section]
\newtheorem{defn}{Definition}[section]

\newtheorem{rem}{Remark}[section]

\def\mb #1{\mathbf{#1}}
\def\wh{\widehat}
\def\bsy{\boldsymbol}

\def\wt{\widetilde}
\def\ch #1{\mathrm{Char}\,#1}
\def\w{\wedge}
\def\sfl{static feedback linearizable}
\def\dfl{dynamic feedback linearizable}
\def\stf{static feedback}
\def\df{dynamic feedback}

\def\wh{\widehat}
\def\bsy{\boldsymbol}

\def\wt{\widetilde}
\def\ch #1{\mathrm{Char}\,#1}
\def\w{\wedge}
\def\mrm #1{\mathrm{#1}}

\def\CV{{\mathcal{V}}}

\def\CW{{\mathcal{W}}}

\def\CB{{\mathcal{B}}}

\def\CP{{\mathcal{P}}}
\def\CK{{\mathcal{K}}}

\def\CL{{\mathcal{L}}}
\def\CO{{\mathcal{O}}}

\def\CS{{\mathcal{S}}}

\def\wt{\widetilde}

\def\b{\beta}
\def\D{\Delta}
\def\e{\varepsilon}

\def\G{\Gamma}

\def\l{\lambda}
\def\L{\Lambda}
\def\k{\kappa}

\def\f{\phi}
\def\vf{\varphi}

\def\o{\omega}

\def\r{\rho}
\def\s{\sigma}
\def\Sg{\Sigma}
\def\t{\tau}
\def\th{\theta}

\def\B#1{\mathbb{#1}}

\def\P #1{\partial_{#1}}

\def\B#1{\mathbb{#1}}

\def\P #1{\partial_{#1}}

\newcommand{\ann}{\mathrm{ann}\,}
\newcommand{\mcal}[1]{\mathcal{#1}}

\newcommand{\der}[2]{\mathcal{#1}^{(#2)}}

\newcommand{\Chare}[1]{\mathrm{Char}\,{\mathcal{#1}}}
\newcommand{\Char}[2]{\mathrm{Char}\,{\mathcal{#1}}^{(#2)}}

\newcommand{\inCharOne}[1]{\mathrm{Char}\,{\mathcal{#1}}^{(1)}_{0}}

\newcommand{\CommentBlock}[1]{\ignorespaces}

\everymath{\displaystyle}

\def\mb #1{\mathbf{#1}}
\def\wh{\widehat}
\def\bsy{\boldsymbol}

\def\wt{\widetilde}
\def\ch #1{\mathrm{Char}\,#1}
\def\w{\wedge}
\def\sfl{static feedback linearizable}
\def\dfl{dynamic feedback linearizable}
\def\stf{static feedback}
\def\df{dynamic feedback}

\def\rk #1{\mathrm{rank} #1}

\def\CV{{\mathcal{V}}}

\def\CB{{\mathcal{B}}}

\def\CP{{\mathcal{P}}}
\def\CK{{\mathcal{K}}}

\def\CL{{\mathcal{L}}}
\def\CO{{\mathcal{O}}}

\def\CS{{\mathcal{S}}}

\def\mcal #1{\mathcal{#1}}
\def\mrm #1{\mathrm{#1}}

\def\wt{\widetilde}

\def\b{\beta}
\def\D{\Delta}
\def\e{\varepsilon}

\def\G{\Gamma}

\def\l{\lambda}
\def\L{\Lambda}
\def\k{\kappa}

\def\f{\phi}
\def\vf{\varphi}

\def\o{\omega}

\def\r{\rho}
\def\s{\sigma}
\def\Sg{\Sigma}
\def\t{\tau}

\def\th{\theta}

\def\B#1{\mathbb{#1}}

\def\P #1{\partial_{#1}}

 	\definecolor{apricot}{rgb}{0.98, 0.81, 0.69}
	\definecolor{aqua}{rgb}{0.0, 1.0, 1.0}
	\definecolor{aquamarine}{rgb}{0.5, 1.0, 0.83}	
	\definecolor{blond}{rgb}{0.98, 0.94, 0.75}	
	\definecolor{blizzardblue}{rgb}{0.67, 0.9, 0.93}	
	\definecolor{cambridgeblue}{rgb}{0.64, 0.76, 0.68}	
	\definecolor{carolinablue}{rgb}{0.6, 0.73, 0.89}	
	\definecolor{champagne}{rgb}{0.97, 0.91, 0.81}	
	\definecolor{cream}{rgb}{1.0, 0.99, 0.82}	
	\definecolor{cornsilk}{rgb}{1.0, 0.97, 0.86}
	\definecolor{deepchampagne}{rgb}{0.98, 0.84, 0.65}
	\definecolor{britishracinggreen}{rgb}{0.0, 0.26, 0.15}
	\definecolor{arsenic}{rgb}{0.23, 0.27, 0.29}
	\definecolor{darksienna}{rgb}{0.24, 0.08, 0.08}
	\definecolor{black}{rgb}{0.00, 0.00, 0.00}
	\definecolor{lavenderpink}{rgb}{0.98, 0.68, 0.82}
	\definecolor{maroon(html/css)}{rgb}{0.5, 0.0, 0.0}
	\definecolor{onyx}{rgb}{0.06, 0.06, 0.06}
	\definecolor{paleblue}{rgb}{0.69, 0.93, 0.93}
	\definecolor{pearl}{rgb}{0.94, 0.92, 0.84}
	\definecolor{paleblue}{rgb}{0.69, 0.93, 0.93}
	\definecolor{smokyblack}{rgb}{0.06, 0.05, 0.03}
	\definecolor{springgreen}{rgb}{0.0, 1.0, 0.5}
	\definecolor{tan}{rgb}{0.82, 0.71, 0.55}
	\definecolor{sunset}{rgb}{0.98, 0.84, 0.65}
	\definecolor{timberwolf}{rgb}{0.86, 0.84, 0.82}
	\definecolor{vanilla}{rgb}{0.95, 0.9, 0.67}
	\definecolor{whitesmoke}{rgb}{0.96, 0.96, 0.96}
	\definecolor{prussianblue}{rgb}{0.0, 0.19, 0.33}
	\definecolor{peach}{rgb}{1.0, 0.9, 0.71}
	\definecolor{white}{rgb}{1.0,1.0,1.0}
	\definecolor{peterGold}{rgb}{1.0,0.95,0.78}
	\definecolor{peterBlue}{rgb}{0.89,1.0,1.0}
	\definecolor{peterRose}{rgb}{1.0,0.94,0.84}
	\definecolor{peterBuff}{rgb}{1.0,1.0,0.84}
\color{black}

\begin{document}

\title{Quotients of Invariant Control Systems}

%\titlerunning{Running Title}

\author{Taylor J. Klotz \and Peter J. Vassiliou}

%\author{Author 1\inst{1}\and Author 2\inst{2}}


\maketitle

%\tableofcontents

\begin{abstract}
In previous work it was shown that if a control system $\CV\subset TM$ on manifold $M$ has a control symmetry group $G$ then it very often has group quotients (or symmetry reductions) $\CV/G$ which are \sfl. This, in turn, can be applied to systematically construct dynamic feedback linearizations of $\CV$ \cite{CKV24}; or to construct partial feedback linearizations \cite{DTVa}, when no dynamic feedback linearization exists. Because of these and related applications, this paper makes a detailed study of  symmetry reduction for control systems. 
We show  that a key property involved in the symmetry reduction of control systems is that of transversality of Lie group actions. Generalizing this notion, we provide an analysis of how the geometry of an invariant distribution, and particularly a control system, is altered as a consequence of symmetry reduction. This provides important information toward understanding the unexpectedly frequent occurrence of \sfl\ quotients.  Specifically, we detail how the integrability properties and ranks of various canonical sub-bundles of the quotient object differ from those of the given distribution. As a consequence we are able to classify the static feedback linearizable quotients of $G$-invariant control systems $\CV$ based upon the geometric properties of $\CV$ and the action $G$. Additionally,  we prove that \stf\ linearizability is preserved by symmetry reduction and the well-known Sluis-Gardner-Shadwick (S-G-S) test for the \stf\ linearization of control systems is extended to  orbital feedback linearization.   A generalized S-G-S test for the {\it a priori} \stf\ and orbital feedback linearizability of $\CV/G$ is also given, based on the Lie algebra of $G$. Finally, we apply all our results to the well-known PVTOL control system. The existence and non-existence of its \sfl\ quotients by 1- and 2 - dimensional Lie groups are classified and it is furthermore shown that the PVTOL can be viewed as an invariant control system on a certain 9-dimensional Lie group.  

\end{abstract}

\section{Introduction}

\subsection{Motivation}

Control systems on Lie groups and more generally control systems with symmetry, form an important and much studied class with many applications. Since these systems generally do not admit static feedback linearization, it is of great interest to study their {\it dynamic} feedback linearization. While the \df\ linearization of control systems has an extensive literature, there appear to be very few papers specifically devoted to the dynamic feedback linearization of control systems with symmetry (see subsection \ref{Relation} for some recent advances).

In this direction, \cite{CKV24} established a sufficient condition in order that a control system with symmetry be \dfl. This sufficient condition is constructive in that, if satisfied, then a systematic procedure enables the construction of a dynamic linearization which makes explicit use of the given symmetry group, $G$. If the control system is represented as a Pfaffian system $\bsy{\o}\subset T^*M$ or equivalently a vector field distribution $\CV\subset TM$ on a manifold $M$, then the procedure in \cite{CKV24} relies upon the existence of a \sfl\ {\it quotient control system} $\bsy{\o}/G$ or equivalently $\CV/G$.  Therefore this paper focuses on the geometric characterization of \sfl\ Lie group quotients $\CV/G$.

An intriguing but currently not well understood phenomenon is that while \sfl\ control systems are rare, \sfl\ {\it quotients} of nonlinear control systems appear to be quite common. Even control systems that are not \dfl\ may yet have \sfl\ quotient control systems (see 5.1 example 1 in \cite{DTVa})! Additionally, if a nonlinear control system $\CV$ is $G$-invariant, we have observed
that $\CV/H$ is very often \sfl\ as $H$ ranges over the Lie subgroups of $G$. In general Lie subgroups $H$ of different dimensions give rise to distinct Brunovsk\'y normal forms for $\CV/H$.

Because of its significance to the problem of \df\ linearization \cite{CKV24}, as well as to partial feedback linearization  \cite{DTVa}, for control systems with symmetry this paper makes a detailed study of symmetry reduction for invariant control systems. In particular, we determine the precise mechanism by which the geometric properties of the quotient control system $\CV/G$ are derived from those of the given invariant control system $\CV$ and the symmetry group $G$ of $\CV$. Importantly, we seek to explain the frequent occurrence of \sfl\ quotients, mentioned above. 

A special case will be of particular interest, namely, the case in which the given control system $\CV$ is itself equivalent to a Brunovsk\'y normal form via some local diffeomorphism of the ambient manifold, including the case when $\CV$ is \sfl. Apart from reasons of completeness, this is because the quotients of \sfl\ control systems turn out to have an important bearing on the general problem of  \df\ linearizability of control systems with symmetry. We shall report on this aspect in a separate work.

\subsection{Outline of Paper}
We remark that the theory of Goursat bundles  \cite{Vassiliou2006a,Vassiliou2006b} implies that any statement about static feedback linearization in this work may be upgraded to orbital feedback linearization, which will be addressrd in forthcoming work. Moreover, many of the results in this paper are of independent interest to the theory of distributions and Pfaffian systems in general.

\begin{itemize}
\item Section \ref{GeomOfControlSysSection} describes the main tools and terminology that are used throughout the paper and may be considered a summary of some of the results of \cite{Vassiliou2006a,Vassiliou2006b}.
\item In Section \ref{controlSymmetriesSect} we give a summary of key results of \cite{DTVa} which defines two key concepts of this work: control admissible symmetries and relative Goursat bundles. The action of such a symmetry (pseudo) group lead to control systems on the associated quotient space so that the number of controls and the time variable are preserved. 
\item Section \ref{WeberAndBryantSect} establishes novel results about Bryant sub-bundles and their role in the geometry of control systems. Moreover, it ties together the notions of Bryant sub-bundles, Weber structures, and Engel rank, some of which have appeared in previous literature. 
\item Section \ref{transversalitySection}, containing one of the main results of the paper, analyses in detail how the geometry of a $G$-invariant  distribution $\CV$ (or Pfaffian system $\bsy{\o}$) differs from that of its quotients $\CV/G$ or $\bsy{\o}/G$, respectively. 
\item In Section \ref{generalLinearizableQuotients} we apply the results of Section \ref{transversalitySection} to arrive at a characterization of when an invariant control system has feedback linearizable quotient control systems. 
\item Section \ref{ReductionOfGoursatBundles} studies symmetry reductions of feedback linearizable control systems. Results from previous sections are used to recall a geometric characterization of \sfl\ control systems. This is applied to show that any quotient of a feeback linearizable control system is also feedback linearizable.
\item In section \ref{GSStestSection}, a simple, rapid test is given for the \stf\ linearization of any control system or any quotient control system 
by generalizing the well-known Sluis-Gardner-Shadwick (S-G-S) test for \stf\ linearization. We also give a general test for the orbital feedback linearization of any control system.
\item Lastly, in Section \ref{PVTOLsection} we apply some of the theory to study the structure of the well-known PVTOL control system. We demonstrate how the results of this paper can be used to predict  the existence and non-existence of some \sfl\ quotients depending only on the dimension and transversality of its symmetry group actions. Finally, we show that the PVTOL system defines an invariant distribution on a 9-dimensional Lie group and we identify the associated Lie algebra explicitly. 
\end{itemize}
The {\tt\bf Maple} package {\tt\bf DifferentialGeometry} \cite{AndersonTorre}, is used for all the computations in this paper.

\subsection{Relation to Other Work}\label{Relation}
Geometric control theory is a vast field of interdisciplinary research. The relationships between nonlinear control theory, differential geometry, and symmetry, in particular the use of Lie brackets to determine feedback invariants, was established and matured during the period of the 1950s through the 1980s, see \cite{BrockettReview} for a historical overview of this time. While a number of researchers have contributed to the current status of static feedback linearizability and feedback invariants over these four decades, of key importance are the results of Roger Brockett \cite{BrockettLin}, and independently of Krener \cite{krener} and of Jacubczyk and Respondek \cite{JakubczykLin}. The aforementioned works essentially state that if a certain collection of distributions -- generated via Lie brackets of vector fields defining a control system -- are all integrable and satisfy simple rank conditions at the origin, then the control system is static feedback linearizable locally about an open set around the origin. These conditions were generalized and expressed in the language of differential forms and Cartan-style differential geometry in \cite{BrunovskySymmetry,GSFeedback,GS92}, and there is also the so-called ``blended algorithm" which uses a combination of differential forms and distribution language \cite{Blended}. 

In this work, we use the theory of Goursat bundles, introduced by the second author in \cite{Vassiliou2006a} and \cite{Vassiliou2006b}. A key reason for this is that it helps to put much of this previous work into a broader geometric context. Importantly, it is not confined to static feedback equivalences but solves the recognition and construction problems for Brunovsky normal forms up to arbitrary local diffeomorphisms; at the same time, it is easy to specialize to \stf\ equivalence as the need arises. This has a direct impact on the problem of orbital feedback linearization, which will be made explicit in future work. Moreover, the integrable bundles that appear in previous  linearization algorithms are effectively constructed from coordinate independent sub-bundles, and hence the theory applies to any smooth vector field distribution or Pfaffian system on a smooth manifold.  

We mention as well the related work of transverse feedback linearization \cite{TVFL1} where one seeks feedback linearizability of the dynamics transverse to a given embedded submanifold that is {\it controlled-invariant} i.e. an invariant submanifold of the ODE system that arises from the control system when the controls are a pure state feedback. In the given controlled-invariant submanifold is part of a family of invariant submanifolds of a control admissible symmetry group action then, in principle, the transverse outputs should match the fundamental functions of the relative Goursat bundles in \cite{DTVa}. Whether invariant submanifolds of control admissible symmetry group actions are controlled-invariant submanifolds is not clear to the authors at the time of writing. 

We mention also recent work concerning the use of symmetry in constructing flat outputs for mechanical control systems. See \cite{KWflatsym}. It would be a worthwhile effort to put these results into the broader geometric picture presented in \cite{CKV24}. 

\section{Geometry of Control Systems} \label{GeomOfControlSysSection}

In this section we review a geometric formulation of control systems expressed in terms of differential geometry on (finite) smooth manifolds. The exposition emphasizes those aspects relevant to the applications that follow. More details can be found in \cite{Vassiliou2006a, Vassiliou2006b, cascade2018, DTVa} and \cite{KlotzThesisShort} (or \cite{KlotzThesisLong} for the full thesis).

\begin{defn}
A {\it control system} is a parametrized family of ordinary differential equations
\begin{equation}\label{controlSystem}
\dot{x}=f(t,x,u),\ \ \ \ \ \ \ \ x\in\B R^n,\ \ \ u\in\B R^m
\end{equation}
in which the vector $x$ is comprised of the {\it state variables} taking values in some open set 
$\mathbf{X}\subseteq \B R^n$ and the vector $u$ is comprised of the {\it inputs} or {\it controls} taking values in some open set $\mathbf{U}\subseteq \B R^m$. Time $t$ takes values in a connected subset of the real line. 

Throughout, we very often invoke the {\it Pfaffian system} representation of a control system as the vanishing
of differential 1-forms 
\begin{equation}\label{controlPfaffian}
\bsy{\o}=\mathrm{span}\Big\{dx^1-f^1(t,x,u)\,dt,\ \ \ dx^2-f^2(t,x,u)\,dt,\ldots,
dx^n-f^n(t,x,u)\,dt\Big\}
\end{equation}
defining a sub-bundle of the cotangent bundle $\bsy{\o}\subset T^*(\B R\times\mathbf{X}\times \mathbf{U})$, and we exploit the geometric properties of $\bsy{\o}$ under local changes of variable. By the same token, we often express our control systems dually as a sub-bundle of the tangent bundle
$\ker\bsy{\o}=\CV\subset T(\B R\times\mathbf{X}\times \mathbf{U})$ 
$$
\CV=\mathrm{span}\big\{\P t+\sum_{i=1}^n f^i(t,x,u)\P {x^i}, \ \ \P {u^1},\ \ \P {u^2},\ \ldots,\ \P {u^m}\Big\},
$$
and frequently switch between the two representations as the need arises. We often refer to 
$\bsy{\o}$ and $\CV$ {\em themselves} as {control systems}.
\end{defn}
\begin{defn}\label{SFT def}
A diffeomorphism $\varphi:M\to N$ of the form 
\[
\varphi:(t,x,u) \mapsto (t,\theta(t,x),\psi(t,x,u)) 
\]
is called a \textit{static feedback transformation} (SFT). Two control systems $(M,\omega)$ and $(N,\eta)$ are called \textit{static feedback equivalent} (SFE) if $\varphi^*\eta=\omega$ for some SFT $\varphi:M\to N$.
\end{defn}
Static feedback linearizable control systems represent a special case of static feedback equivalence, namely, those that are SFE to the Brunovsk\'y normal forms.

Let $M:=\B R\times\mathbf{X}\times\mathbf{U}$. When we wish to draw attention to the state space or control space factors of a manifold $M$ carrying a control system, we write $\mathbf{X}(M)$ or $\mathbf{U}(M)$, respectively.

We now present some definitions and notation from the geometry of distributions that are used in this paper. 
Let us denote by $\CV^{(j)}$ the $j^{\text{th}}$ derived bundle of $\CV:=\CV^{(0)}$,
defined recursively by 
\[\CV^{(j)}=\CV^{(j-1)}+[\CV^{(j-1)},\CV^{(j-1)}],\quad j\geq 1.\]
The sequence
$$
\CV\subset\CV^{(1)}\subset\cdots\subset\CV^{(k)}\subseteq TM
$$
is the {\it\bf derived flag} of $(M, \CV)$ and the integer $k$ is its {\it\bf derived length}. This 
is the smallest integer $k$ such that $\CV^{(k)}=\CV^{(k+1)}$. Throughout we always assume that 
$\CV^{(k)}=TM$.

Denote by $\ch\CV^{(j)}$ the 
{\it\bf Cauchy bundle} of $\CV^{(j)}$, 
$$
\ch\CV^{(j)}=\mathrm{span}\left\{X\in\CV^{(j)}~|~[X, \CV^{(j)}]\subset\CV^{(j)}\right\},\ \ j\geq 0.
$$
We assume for all $j\geq 0$, that $\CV^{(j)}$ and $\ch\CV^{(j)}$ have constant rank and refer to such sub-bundles as {\bf totally regular}. In this totally regular case
$\ch\CV^{(j)}$ can be shown to be integrable for each $j$. Define the {\it\bf intersection bundles}
\begin{equation}\label{intersectionBundleDefn}
\ch\CV^{(i)}_{i-1}:=\CV^{(i-1)}\cap\ch\CV^{(i)},\ \ 1\leq i\leq k-1.
\end{equation}
Unlike the Cauchy bundles, the intersection bundles $\ch\CV^{(i)}_{i-1}$ are not guaranteed to be integrable; however, this will arise as a condition in the definition of Goursat bundle (cf. Definition \ref{Goursat}). 

\begin{defn}\label{decel}
Let $\CV\subset TM$ be a totally regular sub-bundle of derived length $k>1$.
The {\it\bf velocity} of $\CV$ is the ordered list of $k$ integers 
$$
\text{\rm vel}(\CV)=\langle\D_1,\D_2,\ldots,\D_k\rangle,\ \ \text{where}\ \ \D_j=\rk(\CV^{(j)})-\rk(\CV^{(j-1)}),\ 1\leq j\leq k.
$$
The {\it\bf deceleration} of $\CV$ is the ordered list of $k$ integers 
$$
\hskip 20 pt\text{\rm decel}(\CV)=\langle -\D^2_2,-\D^2_3,\ldots,-\D^2_k, \, \D_k\rangle,\ \ \text{where}\ \ \D^2_j=\D_j-\D_{j-1}.
$$
\end{defn}
\begin{defn}
Let 
\begin{equation*}
\begin{aligned}
m_i&=\rk\,\CV^{(i)},\\
\chi^i&=\rk\,\ch\CV^{(i)},\\
\chi^j_{j-1}&=\rk\,\ch\CV^{(j)}_{j-1},\quad 0\leq i\leq k,\quad 1\leq j\leq k-1,
\end{aligned}
\end{equation*}
called the {\it\bf type numbers} of $(M, \CV)$. The list of lists of type numbers
\begin{equation}\label{refinedDerivedTypeDefn}
\begin{aligned}
\mathfrak{d}_r(\CV)=\big[\,\big[m_0,\chi^0\big],\big[m_1,\chi^1_0,\chi^1\big],\big[m_2,\chi^2_1,\chi^2\big],\ldots,\big[m_{k-1},&\chi^{k-1}_{k-2},\chi^{k-1}\big],\cr
&\big[m_k,\chi^k\big]\,\big]
\end{aligned}
\end{equation}
is called the {\it\bf refined derived type} of $(M, \CV)$.
\end{defn}

\subsection{Brunovsk\'y Normal Forms}\label{brun-form-subsect}
In this section we will give an exposition of the partial prolongations of jet spaces and the Brunovsk\'y normal form.
\begin{defn}
Let $J^k(\mathbb{R},\mathbb{R}^m)$ be the standard jet space of order $k$ and let $\bsy{\b}^k_{m}$ be the standard contact system on $J^k(\mathbb{R},\mathbb{R}^m)$ as a Pfaffian system. We will drop the subscript $m$ in the notation $\bsy{\b}^k_m$ and use the shorthand notation $J^k$ when the integer $m$ is known by context. 
\end{defn}

It is proven in \cite{Vassiliou2006a} that the deceleration, 
$\text{decel}(\CV)$ (Definition \ref{decel}), is a diffeomorphism invariant that uniquely identifies the 
Brunovsk\'y normal form of any linearizable control system $\CV$ and we therefore call $\text{decel}(\CV)$ the {\em signature of} $\CV$. The signature of a control system consists of $k$ non-negative integers 
$\text{decel}(\CV)=\langle \r_1, \r_2,\ldots \r_k\rangle$, where $k$ is the derived length of $\CV$. While these are numerical invariants for any control system, if $\CV$ is diffeomorphic to a Brunovsk\'y normal form
then $\r_j$ in $\text{decel}(\CV)$ is the number of sequences of differential forms of order $j$ in the Brunovsk\'y normal form of 
$\mathrm{ann}\CV$. The signature of a control system is widely used in this paper and it is also convenient as a means of classifying Brunovsk\'y normal forms.

\begin{defn}\label{partial prolongation}
A \textit{partial prolongation} of the Pfaffian system 
$(J^1(\mathbb{R},\mathbb{R}^m),\bsy{\b}^{\langle m\rangle})$ 
is a Brunovsk\'y form; i.e., the Pfaffian system associated to the Brunovsk\'y normal form of mixed orders. We use $(J^\kappa(\mathbb{R},\mathbb{R}^m),\bsy{\b}^\kappa)$ to refer to the partial prolongation of {signature} $\kappa=\langle \r_1,\ldots,\r_k\rangle$ where $k$ is the derived length of $\bsy{\b}^\kappa$. 
We also refer to the dual vector distribution $\CB_\k$ as the partial prolongation of $\CB_{\langle m\rangle}$.
\end{defn}
\noindent For example, $\bsy{\b}^{\langle 1, 1\rangle}=\mathrm{span}\{dz-z_1dt, dw-w_1dt, dw_1-w_2dt\}$ 
on $J^{\langle 1, 1\rangle}$ is partial prolongations of 
$\bsy{\b}^{\langle 2\rangle}=\mathrm{span}\{dz-z_1dt, dw-w_1dt\}$, the contact system on $J^1(\B R, \B R^2)$.
Similarly, $\CB_{\langle 1, 1\rangle}=\mathrm{span}\{\P t+z_1\P z+w_1\P w+w_2\P {w_1}, \P {z_1}, \P {w_2}\}$ is a partial prolongation of $\CB_{\langle 2\rangle}$.
\begin{rem}
In the case of a linear control system $\dot{x}=Ax+Bu$, the signature of its distribution representation 
$\CV$ agrees precisely with the
collection of {\it Kronecker indices} of the pair of matrices $(A,B)$. See \cite{CKV24} for more details.
\end{rem}

It is helpful for us to arrange our Brunovsk\'y normal forms according to their signature. In particular, we will think of the new partially prolonged jet space $J^\kappa(\mathbb{R},\mathbb{R}^m)$ as being constructed from jet spaces $J^i(\mathbb{R},\mathbb{R}^{\r_i})$ of fixed order. However, we cannot use a strict product of jet spaces. We must identify the independent variables (i.e. the source) of each jet space together in a product as in, 
\begin{align}
J^\kappa(\mathbb{R},\mathbb{R}^m)&:=\left(\prod_{i\in I}J^i(\mathbb{R},\mathbb{R}^{\rho_i})\right)/\sim,\label{jkappa}\\
\bsy{\b}^\kappa_m&:=\bigoplus_{i\in I}\bsy{\b}^i_{\rho_i},
\end{align}
with $I=\{i\in\B N\,\mid\rho_i\neq 0, 1\leq i\leq k \}$ and each $\bsy{\b}^i_{\r_i}$ is the Brunovk\'y form on jet space $J^i(\mathbb{R},\mathbb{R}^{\r_i})$. The equivalence relation `$\sim$' in (\ref{jkappa}) is defined by
\begin{equation*}
\pi_i\left(J^i(\mathbb{R},\mathbb{R}^{\rho_i})\right)=\pi_j\left(J^j(\mathbb{R},\mathbb{R}^{\rho_j})\right),
\end{equation*}
for all $1\leq i,j\leq k$, where $\pi_i,\pi_j$ are source projection maps (i.e. projection on to the $t$-coordinate on $\mathbb{R}$). 

\begin{exmp}
The Brunovsk\'y normal form of signature $\kappa=\langle1,2,0,0,1\rangle$ on
\begin{equation*}
J^\kappa(\mathbb{R},\mathbb{R}^5)=\left(J^1(\mathbb{R},\mathbb{R})\times J^2(\mathbb{R},\mathbb{R}^2)\times J^5(\mathbb{R},\mathbb{R}))\right)/\sim
\end{equation*}
is generated by the 1-forms
\begin{equation*}
{
\begin{array}{cccc}
\,&\,&\,&\,\\
\,&\,&\,&\theta^4_4=dz^4_4-z^4_5\,dt,\\
\,&\,&\,&\theta^4_3=dz^4_3-z^4_4\,dt,\\
\,&\,&\,&\theta^4_2=dz^4_2-z^4_3\,dt,\\
\,&\theta^2_1=dz^2_1-z^2_2\,dt,&\theta^3_1=dz^3_1-z^3_2\,dt,&\theta^4_1=dz^4_1-z^4_2\,dt,\\
\theta^1_0=dz^1_0-z^1_1\,dt,&\theta^2_0=dz^2_0-z^2_1\,dt,&\theta^3_0=dz^3_0-z^3_1\,dt,&\theta^4_0=dz^4_0-z^4_1\,dt.
\end{array}
}
\end{equation*}
In this example, one can say that $J^\kappa$ has one variable of order 1, two of order 2, zero of orders 3 and 4, and one of order 5. So the signature $\kappa$ represents a list of the local coordinates on $J^\kappa$ categorized by order.
\end{exmp}

The following proposition characterizes the refined derived type \eqref{refinedDerivedTypeDefn} of the 
Brunovsk\'y normal forms. 

\begin{prop}\label{refined derived type numbers}\cite{Vassiliou2006a}
Let $\mcal{B}_\k\subset TJ^\k$ be the distribution that annihilates the 1-forms in a Brunovsk\'y normal form $\bsy{\b}^\kappa$ with signature $\kappa=\langle \rho_1,\ldots,\rho_k\rangle$. Then the entries in the refined derived type 
\begin{equation*}
\mathfrak{d}_r(\mcal{B}_\k)=[[m_0,\chi^0],[m_1,\chi^1_0,\chi^1],\ldots,[m_{k-1},\chi^{k-1}_{k-2},\chi^{k-1}],[m_k,\chi^k]]
\end{equation*}
satisfy the following relations:
\begin{align}
\kappa&=\mathrm{decel}(\mcal{B}_\k),\quad \Delta_i=\sum_{l=i}^{k}\rho_\ell,\\ \nonumber
m_0&=1+m,\quad m_j=m_0+\sum_{l=1}^j\Delta_\ell,\quad1\leq j\leq k,\\ \nonumber
\label{linearTypeConstraints}
\chi^j&=2m_j-m_{j+1}-1,\quad 0\leq j\leq k-1,\\ 
\chi^i_{i-1}&=m_{i-1}-1,\quad 1\leq i\leq k-1, 
\end{align}
where $\Delta_j$ is given in Definition \ref{decel}.
\end{prop}

\subsection{Goursat Bundles}\label{Goursat-subsect}
Here we provide an brief exposition of the theory of Goursat bundles, \cite{Vassiliou2006a,Vassiliou2006b} used in this paper and a discussion of the relevance of this topic to the present study. Certainly Brunovsk\'y normal forms 
$\mcal{B}_\k\subset TJ^\k$ are the local normal forms of Goursat bundles. But in the first instance the theory of Goursat bundles handles the case in which a distribution is equivalent to some Brunovsk\'y normal form via a {\em general diffeomorphism} of the ambient manifolds. 

A Goursat bundle is described as follows. 
\begin{defn}\cite{Vassiliou2006a}\label{Goursat}
A totally regular sub-bundle $\mathcal{V}\subset TM$ of derived length $k$ with $\Delta_k=1$ will be called a {\bf\it Goursat bundle (of signature $\kappa$)} if:
\begin{enumerate}
\item the sub-bundle $\mathcal{V}$ has the refined derived type of a partial prolongation of $J^1(\mathbb{R},\mathbb{R}^m)$ (as characterized in Proposition \ref{refined derived type numbers}) whose signature is
$\kappa=\mathrm{decel}(\mathcal{V})$;
\item each intersection bundle $\Char{V}{i}_{i-1}:= \mathcal{V}^{(i-1)}\cap\Char{V}{i}$ is an integrable sub-bundle, the rank of which agrees with the corresponding rank of $\mathrm{Char}(\mcal{B}_\k)^{(i)}_{i-1}$. That is, intersection bundle ranks satisfy equations \eqref{linearTypeConstraints}.
\end{enumerate}
\end{defn}
In the case $\Delta_k>1$, the full theory of Goursat bundles in \cite{Vassiliou2006a,Vassiliou2006b} requires one to construct an additional bundle, which will be discussed in section \ref{WeberAndBryantSect}.

%The prototypical examples of Goursat bundles are exactly those sub-bundles $\mathcal{B}_\kappa$ of the tangent bundle of some $J^\kappa$ that are annihilated by the Brunovsk\'y 1-forms on $J^\kappa$. 
Theorem \ref{GGNF} below asserts that Goursat bundles are locally diffeomorphic to the Brunovsk\'y normal forms at generic points; and conversely, every Brunovsk\'y normal form is a Goursat bundle. However, this theorem has nothing to say about the singularities of the related {\em Goursat structures} which have been the subject of recent work; see, for example, \cite{MontgomeryGoursatFlags2001} and citations therein. As is the case for the classical Goursat normal form, the generalized Goursat normal form is concerned with generic local behaviour, in terms of which it geometrically characterizes the partial prolongations of the contact system on $J^1(\B R, \B R^m)$ exclusively in terms of their {\it derived type} \cite{BryantThesis}. Theorem \ref{Goursat SFL} gives necessary and sufficient conditions for when this transformation can be chosen to be static feedback (or orbital feedback by replacing $t$ with a suitable $\tau(t,x)$).

\begin{thm}[Generalized Goursat Normal Form]\label{GGNF}\cite{Vassiliou2006a}
Let $\mathcal{V}\subset TM$ be a Goursat bundle on a manifold $M$, with derived length $k>1$, and signature $\kappa=\mathrm{decel}(\mathcal{V})$. Then there is an open dense subset $\mrm U\subset M$ such that the restriction of $\mathcal{V}$ to $\mrm U$ is locally equivalent to $\mcal{B}_\k$ via a local diffeomorphism of $M$. Conversely, any partial prolongation of $(J^1(\B R, \B R^m), \mcal{B}_{\langle m\rangle})$ is a Goursat bundle. 
\end{thm}
The paper \cite{Vassiliou2006a} establishes the local normal form for Goursat bundles constructively. However, in \cite{Vassiliou2006b}, the construction of local coordinates is streamlined into a nearly algorithmic procedure.

%%%%%%%%%%%%%%%%%%%%%%%%%%%%%%%%%%%%%%%%%%%%%%%%%%%%%%%%%%%%%%%%%%%%%%%%%%%%%%%%%

\begin{defn}\cite{Vassiliou2006b}
Let $\mathcal{V}$ be a Goursat bundle of derived length $k$ with $\Delta_k=1$, $\tau$ a first integral of $\Char{V}{k-1}$, and $Z$ any section of $\mathcal{V}$ such that $Z\tau=1$. Then the \textbf{\textit{fundamental bundle}} $\Pi^{k}\subset\mathcal{V}^{(k-1)}$ is defined inductively as
\begin{equation*}
\Pi^{\ell+1}=\Pi^\ell+[\Pi^\ell,Z],\,\,\Pi^0=\inCharOne{V},\,\,0\leq \ell\leq k-1.
\end{equation*}
\end{defn}
The fundamental bundle is sometimes more descriptively referred to as the 
\textbf{\textit{highest order bundle}}. The proof of Theorem 4.2 in \cite{Vassiliou2006a} shows that $\Pi^{k}$ is integrable and has corank 2 in $TM$ while in \cite{Vassiliou2006b} it is proven that in any Goursat bundle, $\ch\CV^{(i)}_{i-1}$ and 
$\Pi^k$ have the form 
\begin{equation}\label{ad char}
\begin{aligned}
\Pi^{(k)}&=\mathrm{span}\{\Pi^0,\mathrm{ad}(Z)\Pi^0,\ldots,\mathrm{ad}(Z)^{k-1}\Pi^0 \},\\
\ch\CV^{(i)}_{i-1}&=\mathrm{span}\{C_0,\mathrm{ad}(Z)C_0,\ldots,\mathrm{ad}(Z)^{i-1}C_0 \},\quad
C_0=\Pi^0,
\quad 1\leq i \leq k-1,
\end{aligned}
\end{equation}
once $Z$ and $\tau$ have been determined.

%We shall be making use of the forms \eqref{ad char} for 
%$\ch\CV^{(i)}_{i-1}$ and $\Pi^k$ in the proof of Theorem \ref{prolongationPredictor} in \S\ref{DFLbySymmetrySection}.

\begin{defn}
The first integrals $\phi^{\ell_j,j}$ of the quotient bundles $\Xi^{(j)}_{j-1}/\Xi^{(j)}$ (where $\Xi^{(j)}_{j-1}=\ann \ch\CV^{(j)}_{j-1}$ and $\Xi^{(j)}=\ann \ch\CV^{(j)}$) are known as the \textit{\textbf{fundamental functions of order}} $j$. We may also refer to the non-$\tau$ first integral of $\Pi^k$ as a fundamental function of order $k$ and we refer to $d\t$ as the {\it independence condition} for the integral curves of $\CV$. 
\end{defn}

The following gives necessary and sufficient conditions under which a Goursat bundle is \sfl.

\begin{thm}[Geometric characterization of \stf\ linearizable control systems]\label{Goursat SFL}
Let $\CV=\mathrm{span}\{\P t+f^i(t,\bsy{x},\bsy{u})\P {x^i},\ \P {u^1},\ldots,\P {u^m}\}$ be a control system defining a totally regular sub-bundle of $TM$. 
%where 
%$M=\B R\times\mathrm{X}(M)\times \mathrm{U}(M)$, $\dim\mathrm{X}(M)=n$, $\dim\mathrm{U}(M)=m$. 
Suppose $(M, \CV)$ is a Goursat bundle of derived length $k>1$ and signature 
$\k=\mathrm{decel}(\CV)$. Then $(M, \CV)$ is locally equivalent to the Brunovsk\'y normal form $\mathcal{B}_\k$ via local diffeomorphism $\varphi: M\to J^\k$. Furthermore, $\varphi$ can be chosen to be an static feedback transformation if and only if 
\begin{enumerate}
%\item[1)] $\ch\CV^{(1)}_0=\mathrm{span}\{\P {u^1},\ldots,\P {u^m}\}$
\item[1)] $dt\in\mathrm{ ann}\,\ch\CV^{(k-1)}$ if $\r_k=1$
\item[2)] $dt\in\mathrm{ann}\,R(\CV^{(k-1)})$ \ \ \ if $\r_k>1$.
\end{enumerate}
\end{thm}

\vskip 3 pt
%%%%%%%
%%%%%%%%%%%%%%%%%%%%%%%%%%%%%%%%%%%%%%%%%%%%%%%%%%%%%%%%%%%%%%%%%%%%%%%%%%

\section{Quotients of Control Systems with Symmetry}\label{controlSymmetriesSect}
In this section we describe the class of Lie group actions that arise in the study of control systems.
Let $\CV$ be a control system on a manifold $M$. A vector field $X$ on $M$ is said to be an {\it infinitesimal symmetry} of $\CV$ if $\CL_X\CV\subseteq\CV$, where $\CL_X$ denotes the Lie derivative with respect to $X$. If $X$ is an infinitesimal symmetry of $\CV$  then the flow $\vf_t$  of $X$ satisfies 
$\big(\vf_t\big)_*\CV=\CV$, $t\in (-\e,\e)$ for some $\e>0$.

A {\it Lie transformation group} $G$ acting on a manifold $M$ is map $\mu:M\times G\to M$ whose elements
$\mu(x,g):=\mu_g(x)$ form a group; in particular, $\mu_h\circ\mu_g(x)=\mu_{gh}(x)$. In this case the elements $\mu_g$ are said to form a so called {\it right-action}. Saying that a control system $\CV$  is invariant under a Lie transformation group $G$ means that $\big(\mu_g\big)_*\CV=\CV$ for all $g\in G$. 

To the {Lie transformation group} $G$ there is an  associated  {\it Lie algebra} 
$\bsy{\G}_\text{alg}$ of $G$ which is a real, finite-dimensional vector space and consists of all those vector fields on $M$ whose flows make up the transformation group $G$. That is, the elements of $\bsy{\G}_\text{alg}$ are infinitesimal generators of $G$. It can be shown that  a connected Lie transformation group $G$ acting on a manifold $M$ leaves a control system $(M, \CV)$ invariant, 
if and only if $\CL_X\CV\subseteq \CV$, for all $X\in \bsy{\G}_\text{alg}$. In this case we say that $G$ is a 
{\it Lie symmetry group} of $\CV$. If $G$ is a Lie transformation group acting on a manifold $M$ and $x\in M$ is a point then  the set $\CO_x=\{y\in M\mid y=\mu_g(x), \forall g\in G\}$ is said to be the {\it orbit} of $G$ through $x$. Under certain circumstances the set of all orbits of $G$ can be given the structure of a smooth manifold 
$M/G$ such that the map $\pi: M\to M/G$, which assigns each point $x\in M$ to its orbit in $M/G$, is a smooth surjective submersion and is, by its definition,  $G$-invariant: $\pi(x\cdot g)=\pi(x)$ for all $g\in G$, where $x\cdot g$ denotes $\mu_g(x)$. To ensure that our space of orbits $M/G$ - also called the quotient of $M$ by $G$ - is a smooth manifold, we shall require our Lie transformation group actions be {\it regular} \cite{Palais}. This is a technical condition, whose details need not concern us, which ensures that each orbit is a closed and embedded submanifold of $M$.

%%%%%%%%%%%%%%%%%%%%%%%%%%%%%%%%%%%%%%%%%%%%%%%%%%%%%%%%%%%%%%%%%%%%%%%%%%

\subsection{Control admissible symmetries}
Let $\mu:M\times G\to M$ be a smooth, regular right action of a Lie group $G$ on a smooth manifold $M$, \cite{OlverSymmetryBook}, \cite{Palais}. Thus the orbit space $M/G$ is a smooth manifold of dimension $\dim M-\dim G$.\footnote{While regularity guarantees the quotient is a smooth manifold it may nevertheless not have the Hausdorff separation property. In this case we restrict to open sets where this holds. For more details see \cite{OlverSymmetryBook} -- \S3.4.} The {\it\bf quotient} of $\CV$ by the action of $G$  is 
{$\CV/G:=d\pi(\CV)$}, where $d\pi$ is the differential of $\pi$. The latter is a distribution on $M/G$, but not necessarily a control system. One can therefore ask:
when is the $G$-quotient of $\CV$ also a control system?
To answer this we describe the appropriate Lie group action. Let $\bsy{\G}_\text{alg}$ be the Lie algebra of infinitesimal generators of the action of $G$ on $M$ and let $\bsy{\G}$ be the sub-bundle of $TM$ generated by $\bsy{\G}_\text{alg}$. See \cite{OlverSymmetryBook} for information on Lie symmetries of differential equations.

\begin{defn}[Control symmetries, \cite{DTVa}]\label{admissibleSymmetries}
Let $\CV$ be a control system on a manifold $M$ and $G$ a Lie group of symmetries of $\CV$. 
Then $G$ is a {\it\bf control symmetry group} if:
\begin{enumerate}
\item $G$ acts regularly on $M$; 
\item the function $t$ is invariant; i.e., 
$\mu_g^*t=t$ for all $g\in G$, and 
\item $\text{rank}\,\big(d\bsy{p}(\bsy{\G})\big)=\dim G$, where $\bsy{p}$ is the projection $\bsy{p}: M\to \B R\times \mathbf{X}(M)$ given by $\bsy{p}(t,x,u)=(t, x)$ and $\bsy{\G}$ is a sub-bundle of $TM$ generated by $\bsy{\G}_\text{alg}$. \footnote{This ensures that the quotient control system will have $\dim\mb X(M)-\dim G$ state variables and that the number of controls will be preserved in the quotient.} 
\end{enumerate}
\end{defn}

The elements of a control symmetry group are static feedback transformations. That is, they have the form (\cite{DTVa}, Theorem 4.9)
\begin{equation*}
\bar{t}=t,\ \ \bar{x}=\th(t,x), \ \ \bar{u}=\psi(t,x,u)
\end{equation*} 
with infinitesimal generators of the form 
$$
X=\xi^i(t,x)\P {x^i}+\chi^a(t,x,u)\P {u^a}\in\bsy{\G}_\text{alg}.
$$
The class of control symmetries is essential for studying the general properties of smooth control systems under the action of a Lie group.

\begin{rem}
{ There is a further subalgebra 
$\bsy{\Sigma}_\text{alg}\subset\bsy{\G}_\text{alg}$ of {\it state-space symmetries} which is better known \cite{GrizzleMarcus, Elkin}. This is the case $\chi^a\equiv 0$ in the infinitesimal generators of $\bsy{\G}_\text{alg}$. But the restriction to $\bsy{\Sigma}$ is both unnecessary and inadequate for studying the full range of phenomena presented by control systems.}
\end{rem}
We can now give criteria whereby the quotient (symmetry reduction) of a control system by a control symmetry group $G$ is also a control system on the quotient manifold $M/G$.

\begin{thm}[\cite{DTVa}]
Let $\mu:M\times G\to M$ be a group of control symmetries of control system 
$(M,\,\CV)$ defined by (\ref{controlSystem}). Let $\bsy{\G}$, the sub-bundle of infinitesimal generators of 
$G$, satisfy
\[ \bsy{\G}\cap\CV^{(1)}=\mathrm{span}\{0\}, \] 
with
$\dim G<\dim\mathbf{X}(M)$. Then the quotient $(M/G, \CV/G)$ is a control system in which 
\[ \dim\mathbf{X}(M/G)=\dim\mathbf{X}(M)-\dim G, \qquad \dim\mathbf{U}(M/G)=\dim\mathbf{U}(M). \]
\end{thm}

\begin{defn}\label{controlAdmissibleDefn}
An action of a Lie group $G$ on the manifold $M$ is \underline{\it\bf control admissible} for a control system $(M, \CV)$ defined by \eqref{controlSystem} if:
\begin{enumerate}
\item $G$ is a control symmetry group of $\CV$; 
\item $\dim G<\dim\mb X(M)$; 
\item The action of $G$ is {\it\bf strongly transverse}, meaning 
$\bsy{\G}\cap\CV^{(1)}=0$.
\end{enumerate}
\end{defn}

\begin{cor}\label{controlAdmissibleCorollary}
Let $(M, \CV)$ be a control system defined by \eqref{controlSystem}. If a Lie group $G$ is control admissible for $(M, \CV)$ then the quotient $\CV/G$ is a smooth control system on the smooth quotient manifold $M/G$ in which $\dim\mb X(M/G)=\dim M-\dim G$ and $\dim \mb U(M/G)=\dim\mb U(M)$. 
\end{cor}

\begin{exmp}\label{MarinoExample} The control system on $\B R^8$ in 5 states and 2 controls,
\begin{equation*}
\CV=\mathrm{span}\{\P t+(x^5x^3+x^2)\P {x^1}+(x^5x^1+x^3)\P {x^2}+u^1\P {x^3}+x^5\P {x^4}+u^2\P {x^5} ,
\P {u^1} ,\P {u^2}\},
\end{equation*}
has a 5-dimensional Lie group of control symmetries. It is easy to check that the subgroup $G$ generated by the Lie algebra
\[ \bsy{\G}_\text{alg}=\mathrm{span}_\mathbb{R}\{X:=x^1\P {x^1}+x^2\P {x^2}+x^3\P {x^3}+u^1\P {u^1}\} \] is control admissible (Definition \ref{controlAdmissibleDefn}). 
%In particular, $\LieD_X\CV\subset\CV$, the infinitesimal counterpart of the invariance condition ${\mu_g}_*
%\CV=\CV$ of Definition \ref{admissibleSymmetries}, where $\LieD_X$ denotes the Lie derivative along $X$.
%It is %known that for any vector fields $X, Y$ on a manifold, $\LieD_XY=[X,Y]$, the Lie bracket of $X$ and %$Y$.
On the $G$-invariant open set $\mrm U\subset \B R^8$ where $x^1\neq 0$, the functions 
\[ t,\ q_1=x^2/x^1,\ q_2=x^3/x^1,\ q_3=x^4, \ q_4=x^5, \ v^1=u^1/x^1,\ v^2=u^2 \] 
are invariant under the action of $G$, which is given by
\[ \bar{x}_1=\e x^1, \ \bar{x}_2=\e x^2,\ \bar{x}_3=\e x^3,\ \bar{x}_4=x^4, \ \bar{x}_5=x^5,\ 
\bar{u}^1=\e u^1,\ \bar{u}^2=u^2, \] 
where $\e\in G$ is an element of the multiplicative group of positive real numbers. If we denote these transformations by $\mu_\e$, then for all $\e\in G$, ${\mu_\e}_*\CV_{|_x}=\CV_{|_{\mu_\e(x)}}$. The functions
$(t, \ q_i, \ v^a)$ form a local coordinate system on an open subset of the quotient manifold 
$M/G$. Furthermore, these functions on $\mrm U\subset\B R^8$ are the components of a local representative of the projection $\pi: \B R^8\to \B R^8/G$ given by,
$\pi_{|_{\mrm U}}(t,x,u)=(t,\ q_i(t,x,u),\ v^a(t,x,u))$. An easy \underline{\tt Maple} computation then gives 
\begin{equation}\label{MarinoQuotient}
\begin{aligned}
d\pi_{|_{\mrm U}}(\CV)&=\CV/G_{|_{\pi(\mrm U)}}=
\mathrm{span}\big\{\P t-(q_1q_2q_4+q_1^2-q_2-q_4)\P {q_1}\\
&-(q_2^2q_4+q_1q_2-v^1)\P {q_2}+q_4\P {q_3}+v^2\P {q_4}, \P {v^1}, \P {v^2}\big\},
\end{aligned}
\end{equation}
a smooth control system on $\pi(\mrm U)\subset M/G$ in accordance with Corollary \ref{controlAdmissibleCorollary}. While $\CV/G$ has 4 states compared to the 5 states of $\CV$, its local form is more complicated, which is typical of a symmetry reduction. However, while $\CV$ is not \sfl, it turns out that \eqref{MarinoQuotient} is \sfl. Thus, the local geometry of $\CV/G$ is radically different from that of 
$\CV$. This fact turns out to have significant consequences in general.
\end{exmp}

\subsection{Relative Goursat bundles}\label{relativeGoursatSect}
Each Brunovsk\'y normal form $\CB_\k$ has trivial Cauchy bundle, $\ch\CB_\k$ $=\mathrm{span}\{0\}$. However, there is an important situation in which a sub-bundle can satisfy all the constraints of a Goursat bundle except for the triviality of its Cauchy bundle. 
\vskip 4 pt
\begin{defn} 
A totally regular sub-bundle $\CV\subset TM$ is a {\it\bf relative Goursat bundle} if it satisfies the requirements of a Goursat bundle (Definition \ref{Goursat}) except for the triviality of its Cauchy bundle. That is, the type number $\chi^0$ need not be equal to zero in a relative Goursat bundle. 
\end{defn}
It is important to note that a relative Goursat bundle has refined derived type satisfying equations (\ref{linearTypeConstraints}).
%Moreover, the quotient of a relative Goursat bundle by its Cauchy bundle is a Goursat bundle.
 
The utility of relative Goursat bundles stems from the ability to very quickly determine the existence of a linearizable quotient $\CV/G$ of a $G$-in\-var\-i\-ant control system $\CV$ without needing to construct 
$\CV/G$ explicitly. In this we always assume that the action of $G$ is control admissible. If its bundle of infinitesimal generators of the action is denoted by $\bsy{\G}$,
then our assumption of strong transversality implies that $\CV^{(1)}\cap\bsy{\G}=0$. We often denote the direct sum $\CV\oplus\bsy{\G}$ by $\wh{\CV}$ and we call $\wh{\CV}$ the {\it\bf augmented system}. 

\begin{thm}[\cite{DTVa}, Theorem 4.5]\label{relGoursatThm}
Suppose that the control system $\CV\subset TM$ has the control admissible Lie group $G$ with sub-bundle of infinitesimal generators $\bsy{\G}$ and satisfies $\ch\CV=0$. 
If $(M, \CV\oplus\bsy{\G})$ is a relative Goursat bundle of derived length $k>1$ and signature
$\k=\mathrm{decel}( \CV\oplus\bsy{\G})$, then there is a local diffeomorphism $\varphi: M/G\to J^\k$ such that 
$\vf_*\left(\CV/G\right)=\CB_\k$.
\end{thm}

An important observation is that even if $\CV$ is not a Goursat bundle, it happens very often that $\CV\oplus\bsy{\G}$ {\it is} a relative Goursat bundle and this can be very significant. However, the local diffeomorphism $\vf$ guaranteed by Theorem \ref{relGoursatThm} need not be a static feedback transformation. To guarantee the existence of such a transformation one imposes slightly more constraints on the relative Goursat bundle. The following generalizes Theorem \ref{relGoursatThm} to the case of {\it \stf\ relative Goursat bundles} and may be regarded as an $``\text{infinitesimal test}"$ for the existence of \sfl\ quotient control systems.

\begin{thm}[\cite{DTVa}]\label{relStatFeedbackLin}
Let the control system $\CV\subset TM$ admit the control admissible Lie group $G$ with sub-bundle of infinitesimal generators $\bsy{\G}$, and satisfy $\ch\CV=0$. Set $\wh{\CV}:=\CV\oplus\bsy{\G}$ and
suppose that $(M, \wh{\CV})$ is a relative Goursat bundle of derived length $k>1$ and signature 
$\k=\mathrm{decel}(\wh{\CV})$. Then the local diffeomorphism $\varphi: M/G\to J^\k$ that identifies 
$\CV/G$ with its Brunovsk\'y normal form
can be chosen to be a static feedback transformation if and only if 
\begin{enumerate}
%\item $\mathrm{span}\{\P {u^1}, \ldots, \P {u^m}\}\subset\ch\wh{\CV}^{(1)}_0$
\item $dt\in\mathrm{ ann}\,\ch\wh{\CV}^{(k-1)}$ if $\Delta_k=1$
\item $dt\in\mathrm{ ann}\,R\left(\wh{\CV}^{(k-1)}\right)$ if $\Delta_k>1$.
\end{enumerate}
\end{thm}

Theorem \ref{relStatFeedbackLin} (\cite{DTVa}, Theorem 4.12) is a geometric characterization of static feedback linearizable quotients of an invariant control system. It is the relative version of Theorem \ref{relGoursatThm}, applied to group quotients. In practice, an invariant control system has many SFL quotients depending on the number of subgroups that satisfy Theorem \ref{relStatFeedbackLin}. 
The final item in the above list, pertaining to the case $\Delta_k>1$, is not required until section \ref{WeberAndBryantSect}, where it will be defined.

\begin{defn}\label{StFrelGoursatDefn}
If a relative Goursat bundle satisfies the hypotheses of Theorem \ref{relStatFeedbackLin} then we call it a 
{\it\bf \stf\ relative Goursat bundle}. 
\end{defn}
An important point is that \stf\ relative Goursat bundles are very easy to identify in practice, once the sub-bundle $\bsy{\G}$ is known. Hence, \sfl\ quotients $\CV/G$ are similarly quickly identified.
\vskip 4 pt
\begin{exmp}\label{ExmpResolventRelativeGoursat}
This example illustrates the various constructions encountered so far. Firstly, we look at a system which is linearizable but {\it not} by \stf\ transformations. %Apart from the notion of a relative Goursat bundle, it illustrates the calculation and use of the resolvent bundle. 

The control system, 
\begin{equation}\label{CharletExample}
\CV=\mathrm{span}\left\{\P t+x^2\P {x^1}+u^1\P {x^2}+u^2\P {x^3}+x^3(1-u^1)\P {x^4}, \P {u^1}, \P {u^2}\right\},
\end{equation}
occurs in \cite{CharletLevineMarino2}. Applying one of the well known tests, for instance, \cite{SluisTilbury} -- Theorem 1 or Theorem \ref{Goursat SFL}, or otherwise, shows that $\CV$ is not \sfl. 

It is easy to see that $X=\P {x^4}$ is an infinitesimal symmetry of $\CV$ defined by (\ref{CharletExample}); that is, $\CL_X\CV\subseteq\CV$. In fact it is an infinitesimal control admissible symmetry. Let us then consider the augmented distribution 
\[ \wh{\CV}:=\CV\oplus\mathrm{span}\{X\}, \] 
as in Theorems \ref{relGoursatThm} and \ref{relStatFeedbackLin}.

While $\wh{\CV}$ is {\it not a control system}, the generalized Goursat normal form can be applied to any smooth sub-bundle $\CV\subset TM$ whose generic solutions are smoothly immersed curves, as in this case. Indeed, we find that the refined derived type of $\wh{\CV}$ is
$$
\mathfrak{d}_r(\wh{\CV})=[[4, 1], [6, 3, 4], [7, 7]].
$$
This is not the refined derived type of a Brunovsk\'y normal form, however it satisfies equations 
(\ref{linearTypeConstraints}) with signature $\text{decel}(\wh{\CV})=\langle 1, 1\rangle$ and derived length $k=2$. Furthermore, we have
exactly one non-trivial intersection bundle $\ch\wh{\CV}^{(1)}_0=\mathrm{span}\{\P {u^1}, \P {u^2}, X\}$ which is integrable, and
\[
\ch\wh{\CV}^{(1)}=\ch\wh{\CV}^{(1)}_0\oplus\mathrm{span}\{\P {x^3}\}.
\] 
Since $\r_2=1$, we check that $t$ is an invariant of
$\ch\wh{\CV}^{(1)}$, and conclude by Theorem \ref{relStatFeedbackLin} that
$\wh{\CV}$ is a {\em \stf\ relative Goursat bundle} of signature $\langle 1, 1\rangle$ which proves that 
$\CV/G$ is static feedback equivalent to $\CB_{\langle 1, 1\rangle}$ where $G$ is the Lie transformation group generated by $X=\P {x^4}$.
\end{exmp}

%%%%%%%%%%%%%%%%%%%%%%%%%%%%%%%%%%%%%%%%

\section{Weber Structures and Bryant sub-bundles}\label{WeberAndBryantSect}
In a Goursat bundle of derived length  $k$, the parameter $\rho_k$ represents the number of 1-forms of highest order $k$ in the Brunovsky normal form $\bsy{\beta}_\k$.  There are two cases to consider according to whether $\rho_k\geq 1$.  The case $\rho_k>1$ requires one to consider the so called resolvent bundle, introduced in \cite{Vassiliou2006a}, which is a sub-bundle of $\CV^{(k-1)}$. The necessity of this construction arises from the fact that the intersection bundle $\ch\CV^{(k)}_{k-1}$, no longer makes sense because $\CV^{(k)}$ spans the whole tangent space so that every vector field in $\CV^{(k)}$ is Cauchy and therefore 
$\ch\CV^{(k)}_{k-1}=\CV^{(k-1)}$. In this section we describe the resolvent bundle and then go on to explain its relation to an alternative geometric characterization of the resolvent, the so called Bryant sub-bundles \cite{BryantThesis}; see also \cite{RespondekContact}.

\begin{defn}\label{struct-tensor-def}
Let $M$ be a smooth manifold and $\CV\subset TM$ a smooth sub-bundle. The structure tensor is the homomorphism of vector bundles $\delta:\L^2\CV\to TM/\CV$
defined by
$$
\delta(X,Y)=[X,Y]\mod\CV,\ \text{for}\ \ X,Y\in\G(M,\CV).
$$
\end{defn}
\noindent {\it\bf The singular variety.}
We can formulate the Cauchy bundles by defining
$$
\s:\CV\to\text{Hom}(\CV,TM/\CV)\ \ \text{by}\ \ \s(X)(Y)=\delta(X,Y).
$$
and assuming $\s$ has constant rank then the Cauchy bundle $\ch\CV$ is its kernel. 

For each $x\in M$, let
$$
\CS_x=\{v\in\CV_x\backslash 0~|~\s(v)\ \text{has less than generic rank}\}
$$
Then $\CS_x$ is the zero set of homogeneous polynomials and so defines a subvariety of the projectivisation 
$\B P\CV_x$ of $\CV_x$. We shall denote by Sing$(\CV)$ the fiber bundle over $M$ with fiber over $x\in M$ equal to 
$\CS_x$ and we refer to it as the {\it singular variety} of $\CV$. For $X\in \CV$ the matrix of the homomorphism $\s(X)$ will be called the {\it polar matrix} of $[X]\in\B P\CV$. There is a map 
$\text{deg}_\CV:\B P\CV\to\B N$ well defined by
$$
\text{deg}_\CV([X])=\text{rank}~\s(X)\ \ \ \text{for}\ \ [X]\in\B P\CV.
$$
We shall call $\text{deg}_\CV([X])$ the {\it degree} of $[X]$. The singular variety $\text{Sing}(\CV)$ is a diffeomorphism invariant in the sense that if $\CV_1,\CV_2$ are sub-bundles over
$M_1,M_2$, respectively and there is a diffeomorphism $\phi: M_1\to M_2$ that identifies them, then 
$\text{Sing}(\CV_2)$ and $\text{Sing}(\phi_*\CV_1)$ are equivalent as projective subvarieties of $\B P\CV_2$. That is, for each $x\in M_1$, there is an element of the projective linear group $PGL({\CV_2}_{|_{\f(x)}},\B R)$ that identifies 
$\text{Sing}(\CV_2)(\f(x))$ and $\text{Sing}(\phi_*\CV_1)$.
\vskip 5 pt
We hasten to point out that the computation of the singular variety for any given sub-bundle $\CV\subset TM$ is algorithmic. That is, it involves only differentiation and commutative algebra operations. In practice, one computes the {determinantal variety} of the polar matrix for generic $[X]$. The {\it determinantal variety}
refers to the subvariety of $\B P\CV$ upon which the polar matrix $\s(X)$ has less than maximal rank, or equivalently, $X$ has less than maximal degree.
\vskip 10 pt

\noindent {\it\bf The singular variety in positive degree.} If $X\in\ch\CV$ then $\text{deg}_\CV([X])=0$. For this reason we pass to the quotient $\widehat{\CV}:=\CV/\ch\CV$. We have structure tensor 
$\widehat{\delta} :\L^2\widehat{\CV}\to \widehat{TM}/\widehat{\CV}$, well defined by
$$
\widehat\delta(\widehat{X},\widehat{Y})=\pi([X,Y])\mod\widehat{\CV},
$$
where $\widehat{TM}=TM/\ch\CV$ and $\pi: TM\to \widehat{TM}$ is the canonical projection. The notion of degree descends to this quotient giving a map $\text{deg}_{\widehat{\CV}}:\B P\widehat{\CV}\to\B N$ well defined by
$$
\text{deg}_{\widehat{\CV}}([\widehat{X}])=\text{rank}~\widehat{\s}(\widehat{X})\ \ \ \ \text{for}\ \ \ \
[\widehat{X}]\in\B P\widehat{\CV},
$$
where $\widehat{\s}(\widehat{X})(\widehat{Y})=\widehat{\delta}(\widehat{X},\widehat{Y})$ for $\widehat{Y}\in\widehat{\CV}$. Note that all definitions go over {\it mutatis mutandis} when the structure tensor $\delta$ is replaced by
$\widehat{\delta}$. In particular, we have notions of polar matrix and singular variety, as before. However, if the singular variety of $\widehat{\CV}$ is not empty, then each point of 
$\B P\widehat{\CV}$ has degree one or more. 
\vskip 10 pt

\begin{defn}[Weber structure] Suppose $\CV\subset TM$ is totally regular of rank $c+q+1$, $q\geq 2, c\geq 0$. Suppose further that $\CV$ satisfies
\begin{enumerate}
\item $\text{rank}\,\CV^{(1)}=c+2q+1$, $\dim\ch\CV=c$
\item 
$\widehat{\Sg}:=\text{Sing}(\widehat{\CV})=\B P\widehat{\CB}\approx\B R\B P^{q-1}$ for some rank $q$ sub-bundle $\widehat{\CB}\subset\widehat{\CV}$ such that each point of $\wh{\Sg}$ has degree 1.
 \end{enumerate}
Then we call $(\CV,\B P\widehat{\CB})$ (or $(\CV,\widehat{\Sg})$) a
{\it Weber structure} of rank $q$ on $M$. 
\end{defn}
\begin{defn}[Resolvent bundle]  
Given a Weber structure
$(\CV,\B P\widehat{\CB})$, let $R(\CV)\subset\CV$, denote the largest sub-bundle such that
\begin{equation}\label{resolventDefn}
{\pi}\big( R(\CV) \big)= \widehat{\CB}.
\end{equation}
We call the rank $q+c$ bundle $R(\CV)$ defined by \eqref{resolventDefn} the {\it resolvent bundle} associated to the Weber structure $({\CV},\widehat{\Sg})$. The bundle $\widehat{\CB}$ determined by the singular variety of 
$\widehat{\CV}$ will be called the {\it singular sub-bundle} of the Weber structure. A Weber structure
will be said to be {\it integrable} if its resolvent bundle is integrable.
\end{defn}
An {\it integrable} Weber structure descends to the quotient of $M$ by the leaves of $\ch\CV$ to be the contact bundle on $J^1(\B R,\B R^q)$. The term honours Eduard von Weber (1870 - 1934) who appears to be the first person to publish a proof of the Goursat normal form (see \cite{Vassiliou2006a}).

\begin{prop}[\cite{Vassiliou2006a}]
Let $(\CV, \B P\wh{\CB})$ be an integrable Weber structure on manifold $M$. Then its resolvent $R(\CV)$ is the unique maximal, integrable sub-bundle of $\CV$.
\end{prop}

We will now review some results of R. Bryant \cite{BryantThesis} that will play a role in this paper; see also \cite{RespondekContact} and \cite{BC3G}. 

\begin{defn}
Let $\CV\subset TM$ be sub-bundle over manifold $M$. A corank 1 sub-bundle $\CB\subset \CV$ will be called a {\it Bryant sub-bundle} or subdistribution if $[\CB, \CB]\equiv 0\mod \CV$.
\end{defn}

\begin{prop}[Bryant]\label{bryant1}
Let $\CV\subset TM$ be a smooth distribution over manifold $M$ of rank $m_0$ and $\CV^{(1)}$ its derived distribution of rank $m_1\geq m_0+1$. Suppose there is a corank 1 subdistribution $\CB\subset \CV$ satisfying 
$[\CB,\CB]\subset \CV$. Then $\ch\CV\subset\CB$ and $\mathrm{rank}\,\ch\CV=2m_0-m_1-1$.
\end{prop}

\begin{proof}
Assume that $\CB\subset\CV$ is a Bryant sub-bundle and let vector field $X$ extend $\CB$ to a basis for 
$\CV$.
Since $[\CB, \CB]\subset\CV$, we have $\CV^{(1)}=\CV+[X, \CB]$ and we observe that since
$\text{rank}\,\CB=m_0-1$ that $\text{rank}\,\CV^{(1)}\leq \text{rank}\,\CV+m_0-1$ and hence
$r:=m_1-m_0\leq m_0-1$.

Suppose there is a nonzero vector field $\xi\in\ch\CV$ that does not belong to $\CB$. This means we can extend $\CB$ by $\xi$ to make a basis for $\CV$ in which case we have that $\CV^{(1)}=\CV+[\xi,\CB]$. But then $[\xi,\CB]\subset\CV$ in which case we obtain the contradiction $\CV^{(1)}=\CV$.
Hence $\ch\CV\subset\CB$.

Suppose that $\text{rank}\,\ch\CV=c$ and spanned by $\{\xi_1,\ldots,\xi_c\}$. Then for some vector fields
$B_1,\ldots B_{m_0-1-c}$ we have
$$
\CB=\mathrm{span}\left\{\xi_1,\ldots,\xi_c\right\}\oplus\mathrm{span}\left\{ B_1,\ldots B_{m_0-1-c}\right\}.
$$
Let $X$ extend $\CB$ to a basis for $\CV$. This means that
$$
\CV^{(1)}=\mathrm{span}\big\{\xi_1,\ldots,\xi_c\big\}\oplus\mathrm{span}\big\{ B_1,\ldots B_{m_0-1-c}\big\}\oplus
\mathrm{span}\big\{[X, B_1],\ldots,[X, B_{m_0-1-c}]\big\}
$$
Since $\text{rank}\,\CV^{(1)}=m_0+r$, the subdistribution $\mathrm{span}\left\{[X, B_1],\ldots,[X, B_{m_0-1-c}]\right\}$ has rank $r$ and hence $m_0-1-c=r$, so that $c=m_0-1-r=2m_0-m_1-1$. 
\end{proof}
Proposition \ref{bryant1} also has a converse, Proposition \ref{bryant2}. For this we need the following
notion
\begin{defn}[Engel rank - {see} \cite{BC3G}, p45]
Let $I$ be a Pfaffian system such that $(d\o)^{\r+1}\equiv 0\mod I$ for all $\o\in I$. Then $\r$ is said to be the {\it Engel rank} of $I$. 
\end{defn}
Proposition 4.1 in \cite{BC3G} is useful for computing Engel rank. If $I$ is spanned by $\o^1,\ldots, \o^s$ and the Engel rank of $I$ is $\r$ then the $(\r+1)$-fold product of the general element in 
$dI\mod I$ is zero. That is, $(t_1d\o^1+t_2d\o^2+\cdots+t_sd\o^s)^{\r+1}\equiv 0\mod I$ for arbitrary $t_i$.

\vskip 2 pt
Note that it is not possible to compute Engle rank by doing so on the individual elements of a basis. 
The following statements are equivalent:

\begin{prop}[Bryant, \cite{BryantThesis}]\label{bryant2}
Let $\CV\subset TM$ be a smooth distribution over manifold $M$ of rank $m_0$ and $\CV^{(1)}$ is its derived bundle of rank $m_1\geq m_0+1$. The following are equivalent:

\begin{enumerate}
\item The Cauchy bundle $\ch\CV$ has rank $2m_0-m_1-1$ and $\mathrm{ann}\,\CV$ has Engel rank equal to 1.
\item The distribution $\CV$ contains a Bryant subdistribution.
\end{enumerate}
\end{prop}

\begin{prop}[Bryant]\label{bryant3}
Let $\CV\subset TM$ be a smooth distribution over manifold $M$ of rank $m_0$ and $\CV^{(1)}$ its derived distribution of rank $m_1\geq m_0+1$. Suppose $\CB\subset \CV$ is a Bryant subdistribution. If $r:=m_1-m_0\geq 3$ then
$\CB$ is integrable.
\end{prop}
\begin{proof} See \cite{RespondekContact}.
\end{proof}

The next result gives another characterization of Bryant sub-bundles that we will use. It asserts that $\CV\subset TM$ contains a Bryant sub-bundle if and only if it determines a Weber structure. 
\begin{thm}[Characterization of Bryant sub-bundles]\label{BryantIffWeber}
Let $q\geq 2$ be integer and $(\CV, \B P\wh{\CB})$ be a Weber structure of rank $q$ on $M$; in particular 
$\mathrm{rank}\CV^{(1)}-\mathrm{rank}\CV=q$. Then a Bryant sub-bundle exists, is unique 
and given by the resolvent bundle of the Weber structure. Conversely, a Bryant sub-bundle of $\CV$
determines a unique Weber structure.
\end{thm}

\begin{proof}
Suppose $(\CV, \B P\wh{\CB})$ is the Weber structure. Let $\{[B_1], [B_2],\ldots,[B_q]\}$ be a basis for $\wh{\CB}$ and $\{\xi_1,\ldots,\xi_c\}$ a basis for
$\ch\CV$. Let $X$ extend the $B_i$ and $\xi_l$ to a basis for $\CV$. Since each line containing $B_i$ has degree 1, there are vector fields $Z_i\in {TM}/{\CV}$ such that 
$Z_1\wedge Z_2\wedge\cdots\wedge Z_q\neq 0$ with
$$
[B_i, {\CV}]\in\mathrm{span}\{ Z_i\}\mod {\CV}, \ \ \ 1\leq i\leq q.
$$
Note that at least one element of the collection $[B_i, {\CV}]$ must be nonzero modulo $\CV$; otherwise $B_i$ would have degree zero and then $B_i\in\ch\CV$, which we exclude. 

For each fixed $i\in\{1, 2, \ldots, q\}$ we have
$$
[B_i, B_j]\in\mathrm{span}\{ Z_i\}\!\!\!\!\!\!\!\mod {\CV}, \ \ \ \ \ \ \ [B_i, X]\in\mathrm{span}\{ Z_i\}
\!\!\!\!\!\!\!\mod {\CV},\ \ \forall j\neq i.
$$
Assuming $[B_i, B_j]\nequiv 0\!\!\!\mod {\CV}$ we deduce that
$Z_1\wedge Z_2\wedge\cdots\wedge Z_q=0$. This contradiction implies that
$[B_i, B_j]\equiv 0\!\!\!\mod {\CV}$ (this implies the sub-bundle whose space of sections is generated by 
$\{[B_1, X], \ldots, [B_q, X]\}$
has rank $q$). 

Let us now observe that $\CB:=\mathrm{span}\{B_1,\ldots, B_q, \xi_1,\ldots, \xi_c\}$ has corank 1 in $\CV$ and satisfies
$[\CB, \CB]\equiv 0\!\!\!\mod \CV$; it is therefore a Bryant sub-bundle. But this is precisely the definition of the resolvent of the Weber structure we started with.

Conversely, suppose $\CB\subset\CV$ is a Bryant sub-bundle. By Proposition \ref{bryant1},
$\CB$ must contain $\ch\CV$ and therefore $\CB$ is spanned by vector fields $B_i, \xi_l$ so that
$\CB=\mathrm{span}\{B_1,\ldots,B_q, \xi_1,\ldots,\xi_c\}$, where $\xi_l\in\ch\CV$. If $X$ extends $\CB$ to a basis for $\CV$ then we have
$[\CB, \CB]\equiv 0\!\!\!\mod \CV$ and $\text{rank}\CV^{(1)}-\text{rank}\CV=q$ and therefore we must have 
$[B_i, X]\in\mathrm{span}\{Z_i\}$ such that $Z_i\in TM/\CV$ where $Z_1\wedge Z_2\wedge\cdots\wedge Z_q\neq 0$. But this means each line in $\mathrm{span}\{B_1,\ldots,B_q\}$ has degree 1 and defines a Weber structure with resolvent equal to $\CB$. 
\end{proof}

\begin{cor}\label{BryantWeberEquivalence}
Let $c\geq 0, q\geq 2$ be integers. Let $\CV\subset TM$ be a sub-bundle over manifold $M$ of rank $m_0=c+q+1$ in which $c=\mathrm{rank}\,\ch\CV$,
and $m_1=\mathrm{rank}\,\CV^{(1)}=c+2q+1$. Then the following are equivalent
\begin{enumerate}
\item $\mathrm{ann}\CV$ has Engel rank 1
\item $\CV$ has a Weber structure of rank $q$
\item $\CV$ has a Bryant sub-bundle
\end{enumerate}
\end{cor}

\begin{proof}
$(1) \Rightarrow (2)$: From Proposition \ref{bryant2} we have that $\CV$ has a Bryant sub-bundle since
$\mathrm{rank}\,\ch\CV=2m_0-m_1-1$. Since $m_1-m_0=q$ then $\CV$ has a Weber structure of rank $q$ by Theorem \ref{BryantIffWeber}. 
$(2) \Rightarrow (3)$: This follows from Theorem \ref{BryantIffWeber}.
$(3) \Rightarrow (1)$: This follows from Proposition \ref{bryant2}.
\end{proof}

This shows that if $\CV\subset TM$ contains a Bryant sub-bundle then it has the general form
\begin{equation}\label{BundleWithWeberStructure}
\CV=\mathrm{span}\{T, \xi_1,\ldots,\xi_c, B_1,\ldots, B_q\}=\mathrm{span}\{T\}\oplus\CB,
\end{equation}
where the $\xi_i$ form a basis for $\ch\CV$, the $B_l$ form a basis for the Weber structure and vector field $T$ extends the Bryant sub-bundle $\CB$ to a basis for $\CV$. 

We now present two examples which illustrate the forgoing discussion of Weber structures and Bryant sub-bundles as well as the generalized Goursat normal form, Theorem \ref{GGNF}, that geometrically characterizes Brunovsky normal forms up to local diffeomorphisms.

\begin{exmp}[cf. (\cite{Vassiliou2006a}, Example 2)]\label{example0}
The control system is
$$
\CV=\mathrm{span}\left\{\P t+(x^2+u^2x^3)\P {x^1}+(x^3+u^2x^1)\P {x^2}+u^1\P {x^3}+u^2\P {x^4}, \P {u^1}, 
\ \P {u^2}\right\},
$$
assumed to be restricted to $M=\{x\in\B R^7~|~x^1x^2\neq0\}$. Its derived length is 2 and the refined derived type is 
$$
\mathfrak{d}_r(\CV)=[[3,0],[5,2,2],[7,7]].
$$
The signature of $\CV$ is $\k=\langle 0, 2\rangle=\langle \rho_1, \rho_2\rangle$ and hence $\rho_2=2>1$. We easily check that $\ch\CV^{(1)}=\mathrm{span}\{\P {u^1},\, \P {u^2}\}$ and hence $t$ is a first integral of $\ch\CV^{(1)}$. However, it can be checked that $\CV$ is not \sfl. The reason is that $\rho_2>1$ and hence the resolvent bundle is to be checked. While the resolvent bundle is integrable it does not have $t$ as one of its invariants. To see this we compute the resolvent bundle $R\big(\CV^{(1)}\big)\subset\CV^{(1)}$ (which agrees with the Bryant sub- bundle $\CB^1$ by Corollary \ref{BryantWeberEquivalence}). We find that
$$
\bar{\CV}^{(1)}:=\CV^{(1)}/\ch\CV^{(1)}=\mathrm{span}\left\{\P t+x^2\P {x^1}+x^3\P {x^2}, \ \P {x^3}, \ 
x^3\P {x_1}+x^1\P {x^2}+\P {x^4}\right\}.
$$
Taking a basis $\{-\P {x^2},\, -x^1\P {x^1}+x^2\P {x^2}\}$ for $\text{Im}\,\wh{\delta}$, where 
$\wh{\delta}:\Lambda^2\bar{\CV}^{(1)}\to TM/\bar{\CV}^{(1)}$ is the structure tensor of $\bar{\CV}^{(1)}$, we obtain the $2\times 3$ polar matrix
$$
P(a):=\left(\begin{matrix}-a_2 & (x^1a_1+x^2a_3)/x^1 & -x^2a_2/x^1\cr 
       -a_3 & a_3/x^1 & a_1-a_2/x^1\end{matrix}\right),
$$
where $a=[a_iY_i]$ is a line field in $\bar{\CV}^{(1)}$ and $Y_1, Y_2, Y_3$ denotes the three basis vectors of 
$\bar{\CV}^{(1)}$. The determinantal variety is easily computed to be
$$
\left\{a\in\B P\bar{\CV}^{(1)}~|~\left\{a_1=(a_2-x^2a_3)/x^1\right\}\cup \left\{a_1=a_2=a_3=0\right\}\right\}.
$$
But for the resolvent bundle we only need the nontrivial component where $P(a)$ has rank 1. We have that $P(a)$ has rank 1 if and only if the $a_j$ are constrained by the equation $a_1=(a_2-x^2a_3)/x^1$,
$a_2^2+a_3^2\neq 0$. This means the singular bundle in $\bar{\CV}^{(1)}$ is 
$$
\Sigma^1:=\mathrm{span}\{Y_1+x^1Y_2,\ x^1Y_3-x^2Y_1\}
$$ 
and therefore the resolvent bundle in $\CV^{(1)}$ (which is also the Bryant sub-bundle) is given by
$$
\CB^1=R\big(\CV^{(1)}\big)=\mathrm{span}\{\P {u^1},\, \P {u^2},\, Y_1+x^1Y_2,\ x^1Y_3-x^2Y_1\}.
$$
This bundle satisfies $[\CB^1, \CB^1]\subseteq\CV^{(1)}$, as it must, and furthermore we find that it is integrable. This verifies that
$\CV$ is locally diffeomorphic to the contact bundle $\CB_{\langle 0, 2\rangle}$ on $J^2(\B R, \B R^2)$ by the generalized Goursat normal, Theorem \ref{GGNF}. But since clearly $t$ is not an invariant of 
$R\big(\CV^{(1)}\big)$, 
$\CV$ is not \sfl\ by Theorem \ref{Goursat SFL}. It is, however, orbital feedback linearizable.
\end{exmp}

\begin{rem}\label{independence condition remark}
If $\CV$ is a Goursat bundle on manifold $M$ then the choice of parameter along the integral curves of $\CV$ is any regular first integral $\t$ of the highest order bundle $\Pi^k$, if $\rho_k=1$ or of the resolvent bundle
$R\big(\CV^{(k-1)}\big)$ if $\rho_k>1$. In either case we shall refer to this choice of first integral as an 
{\it independence condition} for the Goursat bundle $\CV$. 
\end{rem}

The next example illustrates the integrability of the intersection bundles $\ch\CV^{(j)}_{j-1}$  (Definition \ref{Goursat}) as a necessary condition in the characterization of Brunovsky normal forms as proven in Theorem \ref{GGNF}, and its relation to the Engel rank.

\begin{exmp}\label{example1}
The control system
$$
\begin{aligned}
\text{pr}\,\CV:=\mathrm{span}\Big\{\P t+(x^2+u^2x^3)&\P {x^1}+(x^3+u^2x^1)\P {x^2}+(x^4+u^3x^5)\P {x^3}+(u^3x^4+x^5)\P {x^4}\cr
&+u^1\P {x^5}+u^2\P {x^6}+v^1\P {u^2}+v^2\P {v^1}+v^3\P {v^2}, \P {v^3}, \P {u^3},\P {u^1}
\Big\}
\end{aligned}
$$
is bracket generating but not \sfl\ as may easily be checked. It is not even orbital feedback linearizable. In fact, it is {\it not} a Goursat bundle since the intersection bundle $\ch\CV^{(2)}_1$ is not integrable
(Definition \ref{Goursat}) even though its {\it refined derived type} given by
$$
\mathfrak{d}_r(\text{pr}\CV)=[[4, 0], [7, 3, 3], [10, 6, 7], [12, 9, 10], [13, 13]]
$$
is that of a Goursat bundle, since it satisfies equations \eqref{linearTypeConstraints}.  In particular, 
$\chi^2_1=m_1-1=7-1=6$, which is maximal, and additionally $t$ is a first integral of 
$\ch\,\left(\text{pr}\CV^{(3)}\right)$ (Theorem \ref{relStatFeedbackLin}). Nevertheless, by Theorem \ref{GGNF}, $\text{pr}\CV$ cannot be identified with a Brunovsky normal form by {\it any} local diffeomorphism. Note that here the Engel rank of 
$\text{ann}\, \text{pr}\CV^{(1)}$ is 2 rather than 1, as may be checked; compare with Proposition \ref{bryant2}.
 
\end{exmp}

\section{Group Quotients and Transversality}\label{transversalitySection}

In Example \ref{ExmpResolventRelativeGoursat} we studied the control system $\CV$ given by \eqref{CharletExample} relative to its control symmetry group $G$ generated by the infinitesimal control symmetry $\P {x^4}$. In particular, we saw that the refined derived type of the augmented system $\wh{\CV}=\CV\oplus\mathrm{span}\{\P {x^4}\}$ has the form
$$
\mathfrak{d}_r\big(\wh{\CV}\big)=[[4, 1], [6, 3, 4], [7, 7]],
$$
while the refined derived type of $\CV$ itself is given by
$$
\mathfrak{d}_r(\CV)=[[3, 0], [5, 2, 2], [7, 7]].
$$
Furthermore, $\CV$ is not \sfl, while $\wh{\CV}$, and hence $\CV/G$, are \sfl. It is obviously important to understand this transition from a non-\sfl\ control system $\CV$, to a \sfl\ system $\CV/G$, by performing a group quotient. To do this we need to firstly understand precisely how the two refined derived types 
$\mathfrak{d}_r\big(\wh{\CV}\,\big)$ and $\mathfrak{d}_r(\CV)$, are related in terms of fundamental invariants of the derived flags of each sub-bundle $\wh{\CV}$ and $\CV$, respectively. The present section is devoted to solving this problem.
The issue is intimately related to generalizing the notion of strong transversality from Definition 
\ref{controlAdmissibleDefn}. 

\begin{defn}\label{transversalityDefn}
Let $\CV$ be a bracket generating distribution over a manifold $M$ of derived length $k>1$ that is invariant under the action of an $r$-dimensional Lie group $G$ with sub-bundle $\bsy{\G}$ of infinitesimal generators.
We say that the action is $\ell$-{\it transverse for} $\CV$, when $\ell$ is the (necessarily unique or vacuous\footnote{This is for two reasons: (i) the transverse ranks (see \eqref{symmetryTransverseBundles}) are always non-decreasing with each derived system since the symmetries of $\CV$ are also symmetries of any $\CV^{(\ell)}$ for any $\ell$ a non-negative integer and (ii
) $\CV^{(\ell)}\subseteq \CV^{(h)}$ for any $\ell\leq h$}) positive integer such that 
\begin{equation}\label{ellTransversality}
\CV^{(\ell)}\cap\bsy{\G}=\mathrm{span}\{0\}\ \ \ \text{and}\ \ \ \ \CV^{(\ell+1)}\cap\bsy{\G}\neq\mathrm{span}\{0\}.
\end{equation}
Additionally, define the {\it symmetry transverse bundles} by
\begin{equation}\label{symmetryTransverseBundles}
\bsy{\G}_i=\CV^{(i)}\cap\bsy{\G}
\end{equation}
and the {\it transverse ranks} are $r_i=\text{rank}\,\bsy{\G}_i$.
\end{defn}
Notice that in the light of Definition \ref{transversalityDefn}, the notion of strong transversality of Definition 
\ref{controlAdmissibleDefn} coincides with {\it 1-transversality}. The transversality ranks $r_i$ are zero for all
$i\leq \ell$ for an $\ell$-transverse symmetry group action. It is important to point out that a non-trivial
symmetry transverse bundle $\bsy{\G}_i$ may not necessarily contain a Lie subalgebra of $\bsy{\G}$, and, in fact, it need not even be involutive.

To state our next result, we shall define one further sub-bundle. Recall the structure tensor (Definition \ref{struct-tensor-def}) $\delta_i:\L^2\CV^{(i)}\to TM/\CV^{(i)}$
of $\CV^{(i)}$ and let $\pi_i:TM\to TM/\ch\CV^{(i)}$ be the projection to the leaf space $\wt{TM}$ of the integrable constant rank distribution $\ch\CV^{(i)}$. Let us define $\wt{\bsy{\G}}_i=(\pi_i)_*\bsy{\G}_i$
and $\wt{\CV}^{(i)}=(\pi_i)_*\CV^{(i)}$. The structure tensor $\delta_i$ descends to the structure tensor
$\wt{\delta}_i:\L^2\wt{\CV}^{(i)}\to \wt{TM}/\wt{\CV}^{(i)}$ defined by
$\wt{\delta}_i(\wt{X},\wt{Y})=[\wt{X},\wt{Y}]\mod\wt{\CV}^{(i)}$. 
Then, we have
\begin{defn}\label{bundlesKdefn}
The sub-bundles $\CK_i\subset\wt{\CV}^{(i)}$, $0\leq i\leq k-1$ are given by
\begin{equation}\label{bundlesK}
\CK_i=\left\{\wt{X}\in\wt{\CV}^{(i)}~|~\wt{\delta}_i(\wt{X}, \wt{Y})\in\wt{\G}_{i+1}\mod\wt{\CV}^{(i)},
\ \ \forall\ \wt{Y}\in\wt{\CV}^{(i)}\right\}.
\end{equation}
\end{defn}
The sub-bundles $\CK_i$ will be used to relate the refined derived types of $\CV$ and $\wh{\CV}$ by the following result.

\begin{thm}\label{ell-transverseQuotient-refDerTypeTHM}
Let $\CV$ be a bracket generating distribution over an $n$-dimensional manifold $M$ of derived length $k>1$ that is invariant under the action of an $r$-dimensional Lie group $G$ with sub-bundle $\bsy{\G}$ of infinitesimal generators.
Suppose the action is $\ell$-{\it transverse for} $\CV$ for some $0\leq\ell\leq k-1$. Then
\begin{equation}\label{hat{V}^ibundles}
\wh{\CV}^{(i)}=\CV^{(i)}\oplus\L_i,\ \ \ \forall\, i\geq 0\ \ \text{such that}\ \ \mathrm{rank}\,\wh{\CV}^{(i)}<n,
\end{equation}
where $\L_i\subseteq \bsy{\G}$ is any sub-bundle that is transverse to $\bsy{\G}_i$; that is, $\bsy{\G}=\L_i\oplus\bsy{\G}_i$.

Furthermore, the refined derived type of $\wh{\CV}:=\CV\oplus\bsy{\G}$ satisfies
\begin{equation}\label{hatVrefinedDerivedType}
\begin{aligned}
&\wh{m}_i=m_i+r-r_i,\ \ \ \forall\, i\geq 0\ \ \text{such that}\ \ \wh{m}_i<n,\cr
&\wh{\chi}^i=\chi^i+r-p_i+q_i,\ \ \ \forall\, i\geq 0\ \ \text{such that}\ \ \wh{m}_i<n,\cr
&\wh{\chi}^j_{j-1}=\chi^j_{j-1}+r-p'_i+q_i',\ \ \ \forall\, j\geq 1\ \ \text{such that}\ \ \wh{m}_j<n,
\end{aligned}
\end{equation}
where
\begin{equation}\label{bundleRankLabels}
\begin{aligned}
&p_i=\mathrm{rank}\,(\bsy{\G}\cap\ch\CV^{(i)}), \ \ \ \forall\, i\geq 0\ \ \text{such that}\ \ \wh{m}_i<n,\cr
&p_j'=\mathrm{rank}\,(\bsy{\G}_{j-1}\cap\ch\CV^{(j)}),\ \ \ \forall\, j\geq 1\ \ \text{such that}\ \ \wh{m}_j<n\cr
& q_i=\mathrm{rank}\,\CK_i, \ \ \ \forall\, i\geq 0\ \ \text{such that}\ \ \wh{m}_i<n,\cr
&q_j'=\mathrm{rank}(\CK_j\cap\CV^{(j-1)}),,\ \ \ \forall\, j\geq 1\ \ \text{such that}\ \ \wh{m}_j<n.
\end{aligned}
\end{equation}
where $m_i, \chi^i, \chi^j_{j-1}$ are the type numbers of $\CV$ and $\wh{m}_i, \wh{\chi}^i, \wh{\chi}^j_{j-1}$ are those of
$\wh{\CV}$.
\end{thm}
\begin{proof}
We assume that $\CV$ is $\ell$-transverse relative to $\bsy{\G}$, for some $1\leq \ell\leq k-1$. This implies that
 $\wh{\CV}=\CV\oplus \bsy{\G}$ and more generally, for $i\geq 0$, 
$\wh{\CV}^{(i)}=\CV^{(i)}+\bsy{\G}$. Also, for any $i\geq 0$, we denote by $\L_i$ any sub-bundle complementary to $\bsy{\G}_i$, from which it easily follows that $\L_i\cap\CV^{(i)}=0$ which gives 
\eqref{hat{V}^ibundles} and 
$$
\text{rank}\wh{\CV}^{(i)}=\text{rank}\,\CV^{(i)}+\text{rank}\,\L_i,\ \ \ \ \ \ \forall\,i\geq 0\ \ \text{such that}\ \ \wh{m}_i\leq n.
$$
Therefore
$$
\wh{m}_i=m_i+r-r_i, \ \ \ \ \ \ \forall\,i\geq 0\ \ \text{such that}\ \ \wh{m}_i\leq n,
$$
which is the first equation in \eqref{hatVrefinedDerivedType}.
%$$
%\wh{\chi}^j_{j-1}={\chi}^j_{j-1}+r-p'_j+q'_j.
%$$

The Cauchy bundle $\ch\wh{\CV}^{(i)}$ of $\wh{\CV}^{(i)}$ naturally contains both $\ch\CV^{(i)}$
and $\bsy{\G}$ as sub-bundles, but they will often have nontrivial intersections. The ranks of these intersections
are labeled $p_i$, as in \eqref{bundleRankLabels}. Thus, $\mathrm{rank}\,\ch\wh{\CV}^{(i)}$ is at least as large as
$\chi^i+r-p_i$. But, in fact, it may be larger than this because there may be elements $K\in \CV^{(i)}$ such that $K\notin \ch\CV^{(i)}$, but are elements of the form 
\begin{equation}\label{extraCauchyVectors}
\wh{C}=K+f^aX_a\ \ \ \text{where}\ \ \ \ X_a\in\bsy{\G}, 
\end{equation}
which belong to $\ch\wh{\CV}^{(i)}$, for some smooth functions $f^a$ on $M$. In other words a vector field $K\in\CV^{(i)}$ that is not an element of 
$\ch\CV^{(i)}$ may be 
$``\text{modified}"$ by elements of $\bsy{\G}$ so as to be a Cauchy vector for $\wh{\CV}^{(i)}$. 
Besides $\bsy{\G}$ and $\ch\CV^{(i)}$, this is the only other possible source of Cauchy vectors for 
$\wh{\CV}^{(i)}$. Suppose $K$ is such a vector field. To avoid triviality we choose
$Z\in\CV^{(i)}$ such that ${\delta}_i(K,Z)\neq 0$, where ${\delta}_i$ is the structure tensor of ${\CV}^{(i)}$. We will show that $K\in\CK_i$ given by \eqref{bundlesK} in Definition \ref{bundlesKdefn}.
A calculation shows that for $\wh{C}\in\ch\wh{\CV}^{(i)}$, it is necessary and sufficient that $[K,Z]\in \wh{\CV}^{(i)}$. That is, $[K,Z]$ is a symmetry of $\wh{\CV}^{(i)}$. But since
$\delta_i(K,Z)\neq 0$ we have $[K,Z]\in\CV^{(i+1)}$ and $[K,Z]\notin\CV^{(i)}$. It follows that
$[K,Z]\in\bsy{\G}_{i+1}\mod \CV^{(i)}$. Since $K\notin\ch\CV^{(i)}$, we can characterize $K$ as belonging to the bundle $\bsy{\wt{\G}}_{i+1}/\wt{\CV}^{(i)}$; in other words, $K\in\CK_i$. This proves the second formula in \eqref{hatVrefinedDerivedType},
$$
\wh{\chi}^i=\chi^i+r-p_i+q_i,\ \ \ \forall\, i\geq 0\ \ \text{such that}\ \ \wh{m}_i<n,
$$
according to the definition of $q_i$ in \eqref{bundleRankLabels}. 

To establish the last formula in \eqref{hatVrefinedDerivedType} we recall that
$\ch\wh{\CV}^{(j)}_{j-1}=\wh{\CV}^{(j-1)}\cap\ch\wh{\CV}^{(j)}$
and note that it must contain $\bsy{\G}$ and $\ch\CV^{(j)}_{j-1}$ as sub-bundles. But their union will be linearly dependent because we have double counted the bundle 
$\bsy{\G}_{j-1}\subset\CV^{(j-1)}\cap\ch{\CV}^{(j)}$. Hence
$\wh{\chi}^j_{j-1}\geq\chi^j_{j-1}+r-p_j'$. Finally, as before there may be vector fields 
$K\in\CV^{(j-1)}$ that give rise to elements of the form \eqref{extraCauchyVectors} that also belong to
$\CK_j$ and contained in $\ch\wh{\CV}^{(j)}_{j-1}$. This leads to the final formula in
\eqref{hatVrefinedDerivedType} as per the definitions in \eqref{bundleRankLabels}.
\end{proof}

A particularly nice conclusion that can be drawn from Theorem \ref{ell-transverseQuotient-refDerTypeTHM}
is that we can deduce the deceleration of $\wh{\CV}$, and hence $\CV/G$, immediately in terms of the 
$``\text{intersection deceleration}"$ of $\bsy{\G}$ relative to $\CV$.

\begin{cor}\label{decelVhat}
Let $\CV$ be a bracket generating distribution over an $n$-dimensional manifold $M$ of derived length $k>1$ that is invariant under the action of an $r$-dimensional Lie group $G$ with sub-bundle $\bsy{\G}$ of infinitesimal generators.
Suppose the action is $\ell$-{\it transverse relative to} $\CV$ for some $1\leq\ell\leq k-1$. 
Then the deceleration of $\wh{\CV}$ is given by
\begin{equation}
\mathrm{decel}\left(\wh{\CV}\right)=
\left\langle \nabla^2_2-\Delta^2_2,\ 
\nabla^2_3-\Delta^2_3,\ldots, \nabla^2_{\wh{k}}-\Delta^2_{\wh{k}}, 
\Delta_{\wh{k}}-\nabla_{\wh{k}}\right\rangle,
\end{equation}
where $\nabla_i=r_i-r_{i-1}, \nabla_i^2=\nabla_i-\nabla_{i-1}$, and $\wh{k}$ is the derived length of 
$\wh{\CV}$.
\end{cor}

Notice that Corollary \ref{decelVhat} strongly constricts the transversality ranks.

\begin{cor}\label{tot-transverse-ranks-cor}
Let $\CV$ be a bracket generating distribution over an $n$-dimensional manifold $M$ of derived length $k>1$ that is invariant under the action of an $r$-dimensional Lie group $G$ with sub-bundle $\bsy{\G}$ of infinitesimal generators.
Suppose the action is $\ell$-{\it transverse for} $\CV$ for some $1\leq\ell\leq k-1$. If
$r=\sum_{j=\ell+1}^k\Delta_j$ then the derived length of $\wh{\CV}$ is $\ell$ and
$$
\begin{aligned}
&\wh{m}_i=m_i+r, \ \ 0\leq i\leq\ell-1,\cr
&\wh{\chi}^i=\chi^i+r, \ \ 0\leq i\leq\ell-1,\cr
&\wh{\chi}^j_{j-1}=\chi^j_{j-1}+r, , \ \ 1\leq j\leq\ell-1.
\end{aligned}
$$
\end{cor}

\begin{proof}
The condition on the dimension $r$ of the symmetry group $G$ together with $\ell$-transversality forces 
$\bsy{\G}$ to complete any frame of sections of $\CV^{(\ell)}$ to a framing of $TM$, resulting in the derived length of $\wh{\CV}$ being $\ell$. Then the refined derived type integers $\wh{m}_i, \wh{\chi}^i, \wh{\chi}^j_{j-1}$ follow from Theorem \ref{ell-transverseQuotient-refDerTypeTHM}.
\end{proof}

\begin{exmp} 
We now provide an example illustrating Theorem \ref{ell-transverseQuotient-refDerTypeTHM}.
\vskip 6 pt
\noindent{\bf Control system:}
$$
\begin{aligned}
\CW:=\mathrm{span}\Big\{X:=\P t+\left(x^1u^1+x^2\right)\P {x^1}&+u^2\P {x^2}+\left(u^1(x^1)^2+x^1x^2+x^4\right)\P {x^3}\cr
&+x^5\P {x^4}+x^6\P {x^5}+x^7\P {x^6}+u^1\P {x^7},\ \P {u^1},\ \P {u^2}\Big\}
\end{aligned}
$$
\noindent{\bf Symmetries:}\newline
$$
\begin{aligned}
&Y_1=x^1\P {x^1}+x^2\P {x^2}+(x^1)^2\P {x^3}+u^2\P {u^2}\cr
&Y_2=t\P {x^3}+\P {x^4}
\end{aligned}
$$
\noindent {\bf Refined derived type:}\newline
$$
\mathfrak{d}_r(\CW)=[[3, 0], [5, 2, 2], [7, 4, 5], [8, 6, 6], [9, 7, 7], [10, 10]]
$$
Let $\bsy{\G}$ be generated by $Y_1, Y_2$. Then we get
$$
\begin{aligned}
&\mathfrak{d}_r(\CW\oplus \Gamma
)=[[5, 2], [7, 4, 5], [8, 6, 6], [9, 7, 7], [10, 10]]
\end{aligned}
$$
\noindent {\bf Transverse bundles:}\ \   $\G_i=:\G\cap \CW^{(i)}$, $\G=\mathrm{span}\{Y_1, Y_2\}${\bf :}  \hfill $(r_i=\text{rk}\,\G_i)$
$$
\G_0=\G_1=\mathrm{span}\{0\}, \ \ \G_2=\G_3=\G_4=\mathrm{span}\{Y_1\}
$$
\noindent {\bf Bundles:} $\CP_i:={\G \cap \ch\CV^{(i)}}$:\hfill  $(p_i=\text{rk}\,\CP_i)$\newline
$$
\CP_0=\CP_1=\mathrm{span}\{0\},\ \ \CP_2=\CP_3=\CP_4=\mathrm{span}\{Y_1\}
$$
\noindent {\bf Bundles:} $\CP_j':={\G_{j-1}\cap \ch\CV^{(j)}}$:\hfill  $(p'_i=\text{rk}\,\CP'_i)$\newline
$$
\CP_1'=\CP_2'=\mathrm{span}\{0\},\ \ \CP_3'=\mathrm{span}\{Y_1\}
$$

\noindent {\bf Cauchy bundles:}
$$
\begin{aligned}
&\ch\CW=\mathrm{span}\{0\}\cr
&\ch\CW^{(1)}=\mathrm{span}\{\P {u^1}, \P {u^2}\}\cr
&\ch\CW^{(2)}=\ch\CW^{(1)}\oplus\mathrm{span}\left\{\P {x^2}, \P {x^1}+x^1\P {x^3}, \P {x^7}\right\}\cr
&\ch\CW^{(3)}=\ch\CW^{(2)}\oplus\mathrm{span}\left\{x^2(\P {x^1}+x^1\P {x^3})-\P {x^6}\right\}\cr
&\ch\CW^{(4)}=\ch\CW^{(3)}\oplus\mathrm{span}\left\{(x^2u^1-u^2)(\P {x^1}+x^1\P {x^3})-\P {x^5}\right\}
\end{aligned}
$$
\noindent {\bf Bundles $\CK_i$:}\ \hfill $(q_i=\text{rk}\,\CK_i)$\newline
Let $\wt{\CW}^{(i)}=\CW^{(i)}/\ch\CW^{(i)}$, \ \ $\wt{\G}_i=\pi_i\left(\G\right)$, where 
$\pi_i: TM\to TM/\ch\CW^{(i)}$. Let $\wt{\delta}_i:\L^2\wt{\CW}^{(i)}\to \wt{TM}/\wt{\CW}^{(i)}$ be the structure tensor for each $i$ and 
$$
\wt{\s}_i: \wt{\CW}^{(i)}\to \text{Hom}\left(\wt{\CW}^{(i)}, \wt{TM}/\wt{\CW}^{(i)}\right)
$$ 
be defined by $\wt{\s}_i(X)(Y)=\wt{\delta}_i(X,Y)$. Let
$$
\CK_i=\left\{X\in\wt{\CW}^{(i)}~|~\wt{\s}_i(X)(Y)\in\wt{\G}_{i+1}\mod\wt{\CW}^{(i)},\ \ 
\forall \ Y\in\wt{\CW}^{(i)}\right\}.
$$
We need the $\wt{\CW}^{(i)}$:
$$
\begin{aligned}
&\wt{\CW}^{(0)}=\CW\cr
&\wt{\CW}^{(1)}=\mathrm{span}\{X\}\oplus\mathrm{span}\{\P {x^2}, x^1(\P {x^1}+x^1\P {x^3})+\P {x^7}\}\cr
&\wt{\CW}^{(2)}=\wt{\CW}^{(1)}\oplus\mathrm{span}\{\P {x^1}+x^1\P {x^3},\ \P {x^6}\}\cr
&\wt{\CW}^{(3)}=\wt{\CW}^{(2)}\oplus\mathrm{span}\left\{\P {x^5}\right\}\cr
&\wt{\CW}^{(4)}=\wt{\CW}^{(3)}\oplus\mathrm{span}\left\{\P {x^4}\right\}\cr
&\wt{\CW}^{(5)}=\wt{\CW}^{(4)}\oplus\mathrm{span}\left\{\P {x^3}\right\}
\end{aligned}
$$
and the $\wt{\G}_{i+1}=\pi_i(\G_{i+1})$:
$$
\begin{aligned}
&\wt{\G}_1=\mathrm{span}\{0\}\cr
&\wt{\G}_2=\mathrm{span}\{x^1\P {x^1}+x^2\P {x^2}+(x^1)^2\P {x^3}\}\cr
&\wt{\G}_3=\mathrm{span}\{0\}\cr
&\wt{\G}_4=\mathrm{span}\{0\}\cr
\end{aligned}
$$

It is easy to show that $\wt{\s}(X)$ is the zero map if and only if $X$ is a Cauchy vector. Hence if
$\wt{\G}_i\mod\wt{\CW}^{(i)}$ is trivial then $\CK_i$ will be nontrivial if and only if $X$ is a Cauchy vector. 
Since $\ch\wt{\CW}^{(0)}=\mathrm{span}\{0\}$  we have  $\CK_0=\mathrm{span}\{0\}$.  

Next 
$$
\wt{\G}_2\mod \wt{\CW}^{(1)}=\mathrm{span}\left\{x^1(\P {x^1}+x^1\P {x^3})\right\}
$$
and we have $X, \P {x^2}\in\wt{\CW}^{(1)}$.  
Since $[\P {x^2}, X]=\P {x^1}+x^1\P {x^3}\in \wt{\G}_2\mod \wt{\CW}^{(1)}$, we get
$$
\CK_1=\mathrm{span}\{\P {x^2}\},
$$
since no other elements $K$ of $\wt{\CW}^{(1)}$ have an image in $\wt{\G}_2$ under $\wt{\s}_1(K)$.

The remaining $\CK_i$ are trivial because $\wt{\G}_3, \wt{\G}_4$ are trivial and the $\wt{\CW}^{(j)}$ contain no
Cauchy vectors by construction.
\vskip 10 pt
\noindent{\bf Bundles} $\CK_j\cap\CW^{(j-1)}$\ \hfill $(q_j'=\text{rk}\,(\CK_i\cap\CW^{(j-1)}))$ \newline
By inspection we see that  all these bundles are trivial over the allowed range  $j=1, 2, 3$.
\vskip 10 pt
\noindent{\bf Bundles} $\G_{j-1}\cap\ch\CW^{(j)}$\ \hfill $(p'_j=\text{rk}\,(\G_{j-1}\cap\ch\CW^{(j)}))$\newline
From the data given above, we see that these bundles are also trivial  over the allowed range  $j=1, 2, 3,4$.

This completes the construction of all the bundles necessary to compute the refined derived type of 
$\CW/G$. One can now compare the refined derived types using Theorem \ref{ell-transverseQuotient-refDerTypeTHM}.
\end{exmp}

Notice that the control system in the above example is Goursat and so is its augmented bundle. This is no coincidence, see Theorem \ref{GoursatQuotientsRGoursat}.

\section{Linearizable Quotients of Invariant Control Systems}\label{generalLinearizableQuotients}
Let $\CV\subset TM$ be a sub-bundle invariant under the regular action of a Lie group $G$ with sub-bundle of infinitesimal generators 
$\bsy{\G}$ and assume that $G$ acts strongly transitively on $M$ relative to $\CV$; that is, $\CV^{(1)}\cap\bsy{\G}=\mathrm{span}\{0\}$. 
In this section we establish a characterization of those $G$-invariant sub-bundles whose quotients $\CV/G$ by the action of $G$ are locally diffeomorphic to some partial prolongation of the contact system on the jet space 
$J^1(\B R,\B R^m)$, $m\geq 2$. As an application, we characterize those $G$-invariant control systems $\CV$
that have \sfl\ quotients $\CV/G$. Recall that interest in \sfl\ quotients arises from the fact that they can be used to construct dynamic feedback linearizations of given control system $\CV$, \cite{CKV24}. 
\subsection{Characterization of linearizable quotients}
We start this subsection with a simple lemma concerning the maximal rank of the intersection bundles $\ch\CV^{(j)}_{j-1}$ of a totally regular bracket generating distribution $\mcal{V}$. 
\begin{lem}\label{maxIntersections}
Let $\CV\subset TM$ be a totally regular bracket generating sub-bundle over manifold $M$ of derived length $k>1$.
If $0<j<k$ then the maximum rank of the intersection bundle $\ch\CV^{(j)}_{j-1}:=\CV^{(j-1)}\cap\ch\CV^{(j)}$ is $m_{j-1}-1$ where 
$m_j:=\text{rank}\,\CV^{(j)}$. If, in addition, $\mathrm{rank}\,\ch\CV^{(i)}=2m_i-m_{i+1}-1$, $0\leq i<k$ then $\CV$ has the refined derived type of a Goursat bundle.
\end{lem}

\begin{proof}
Suppose on the contrary $\chi^j_{j-1}:=\text{rank}\,\ch\CV^{(j)}_{j-1}>m_{j-1}-1$ in which case $\chi^j_{j-1}=m_{j-1}$;
and this means $\CV^{(j-1)}\subseteq\ch\CV^{(j)}$. If $\CV^{(j-1)}\neq\ch\CV^{(j)}$ then we can complete
$\CV^{(j-1)}$ by taking Lie brackets so that $\CV^{(j-1)}+[\CV^{(j-1)},\CV^{(j-1)}]\subseteq\ch\CV^{(j)}$ since the latter is integrable. But the left-hand-side of this equation is precisely the definition of $\CV^{(j)}$ and this implies that $\CV^{(j)}$ is integrable. Since $j<k$ this contradicts the hypothesis that the derived length is $k$, which proves the lemma. Finally, if $\mathrm{rank}\,\ch\CV^{(i)}=2m_i-m_{i+1}-1$ then $\CV$ has refined derived type of a relative Goursat bundle by \cite{Vassiliou2006a}, equation (3.5).
\end{proof}

The import of Lemma \ref{maxIntersections} is that choosing (when possible) the symmetry algebra $\bsy{\G}$ such that intersection bundles have maximal rank then reduces the check for linearizable quotient control systems to checking the integrability of certain canonical bundles. This notion is first applied to a general sub-bundle in Theorem \ref{linearizableQuotientExGoursat} and subsequently applied to the special case of control systems.

\begin{thm}\label{linearizableQuotientExGoursat}
Let $\CV\subset TM$ be a sub-bundle over manifold $M$ and $G$ its Lie group of symmetries acting regularly and strongly transitively relative to $\CV$ with sub-bundle of infinitesimal generators 
$\bsy{\G}$. Suppose that 
$\wh{\CV}:=\CV\oplus\bsy{\G}\subset TM$ is totally regular of derived length $\wh{k}$  such that the intersection bundle
$\ch\wh{\CV}^{(j)}_{j-1}$ is maximal and integrable for each $j$. If $\Delta_{\wh{k}}>1$, suppose $\wh{\CV}^{(\wh{k}-1)}$ contains an integrable Weber structure. Then the quotient $(M/G, \CV/G)$ is locally diffeomorphic to a Brunovsky normal form.
\end{thm}

\begin{proof}
As described in Lemma \ref{maxIntersections}, we assume that $\bsy{\G}$ is such that 
$\ch\wh{\CV}^{(j)}_{j-1}$ has maximal rank for each $j$. Since $\ch\wh{\CV}^{(j)}_{j-1}$ is integrable and has corank 1 in 
$\CV^{(j)}$ it is a Bryant sub-bundle and by Proposition \ref{bryant2}, $\ch\wh{\CV}^{(j-1)}$ has rank
$2m_{j-1}-m_j-1$. 

Now it can happen that $\ch\wh{\CV}^{(j)}_{j-1}=\ch\wh{\CV}^{(j)}$. In that case the intersection bundle is automatically integrable and continues to have corank 1 in $\CV^{(j-1)}$. In this case, 
$\ch\wh{\CV}^{(j)}$ also forms a Bryant sub-bundle in $\CV^{(j-1)}$ and therefore
$\text{rank}\,\ch\wh{\CV}^{(j)}=2m_{j-1}-m_j-1$ by Proposition \ref{bryant2}. So far we have considered 
the derived flag for $0\leq j\leq\wh{ k}-1$ and we have shown that $\wh{\CV}$ has the refined derived type of a Goursat bundle.

Therefore we can use the generalized Goursat normal form,  Theorem \ref{GGNF}, to complete the proof. If 
$\Delta_{\wh{k}}>1$ then $\wh{\CV}^{(\wh{k}-1)}$ has an integrable Weber structure and therefore an integrable resolvent bundle $R\big(\wh{\CV}^{(\wh{k}-1)}\big)$. On the other hand if $\Delta_{\wh{k}}=1$ then
we can form the highest order bundle $\Pi^{\wh{k}}$ which is integrable. This data, together with the of the refined derived type of $\wh{\CV}$ being that of a relative Goursat bundle, we conclude that $\wh{\CV}$ is a relative Goursat bundle and by Theorem \ref{relGoursatThm}, the quotient $(M/G, \CV/G)$ is a Goursat bundle. This means that
there is a local diffeomorphism which identifies $\CV/G$ with the contact distribution on a partial prolongation of the jet space $J^1(\B R, \B R^m)$, which is a Brunovsky normal form.
\end{proof}

Theorem \ref{linearizableQuotientExGoursat} is easily adapted to describe the necessary and sufficient conditions for the local diffeomorphism to be
a \stf\ transformation in case $\CV$ is a control system.

\begin{cor}\label{stf_linearizableQuotientExGoursat}
Let $(M, \CV)$ be a control system and $G$ its Lie group of control admissible symmetries with sub-bundle of infinitesimal generators
$\bsy{\G}$. Suppose 
$\wh{\CV}:=\CV\oplus\bsy{\G}\subset TM$ satisfies the hypotheses of Theorem \ref{linearizableQuotientExGoursat} and that
 if $\rho_{\,\wh{k}}=1$, $\ch\wh{\CV}^{(\,\wh{k}-1)}$ has invariant $t$; and if $\rho_{\,\wh{k}}>1$, the resolvent has invariant $t$. Then the quotient $(M/G, \CV/G)$ is
\sfl.
\end{cor}
\begin{proof}
Apply Theorem \ref{relStatFeedbackLin} to show that the local diffeomorphism guaranteed by Theorem \ref{linearizableQuotientExGoursat} can be chosen to be a \stf\ transformation.
\end{proof}

\begin{cor}\label{uniformGoursat}
Let $(M, \CV)$ be a control system and $G$ its Lie group of control admissible symmetries with sub-bundle of infinitesimal generators $\bsy{\G}$. Suppose $G$ is such that 
$\wh{\CV}:=\CV\oplus\bsy{\G}\subset TM$ has the refined derived type of a Goursat bundle with signature of the form
$\langle 0, 0, \ldots,0, \r_{\,\wh{k}}\rangle$, where $\r_{\,\wh{k}}>1$ and $\wh{k}$ is the derived length of $\wh{\CV}$. Then
$\CV/G$ will be \sfl\ if and only if the resolvent exists, is integrable and has invariant $t$. If 
$\r_{\,\wh{k}}=1$ then $\CV/G$ will be \sfl\ if and only if $t$ is an invariant of $\ch\wh{\CV}^{(\,\wh{k}-1)}$.
\end{cor}
\begin{proof}
If $\wh{\CV}$ has the hypothesized signature, it follows that intersection bundles and Cauchy bundles agree and are integrable. Then by the generalized Goursat normal form \cite{Vassiliou2006a, Vassiliou2006b}, the only obstruction to the \stf\ linearizability of $\CV$ is the existence and integrability of the resolvent bundle and that $t$ is one of its first integrals in the case $\rho_{\,\wh{k}}>1$. If $\rho_{\,\wh{k}}=1$ then one can use the classical Goursat normal form together with the hypothesis that $t$ is a first integral of 
$\ch\wh{\CV}^{(\,\wh{k}-1)}$.
\end{proof}

\section{Quotients of Linearizable Control Systems}\label{ReductionOfGoursatBundles}
In this section we study symmetry reductions of feedback linearizable control systems. Theorem \ref{GoursatIFFbryant} characterizes Goursat bundles in terms of the existence of Bryant sub-bundles and this is then used to prove that the quotient of any Goursat bundle is a Goursat bundle. We also prove that the refined derived type of the quotient can be determined from the refined derived type of the augmented distribution $\wh{\CV}=\CV\oplus\bsy{\G}$ (Theorem \ref{RDTofQuotient}).

\begin{lem}\label{BryantBundlesTelescope}
Let $\CV\subset TM$ be a totally regular sub-bundle over manifold $M$ with derived flag
$\CV=\CV^{(0)}\subset\CV^{(1)}\subset \cdots\subset\CV^{(k)}$ of derived length $k$. If
each term in the derived flag contains a Bryant sub-bundle, $\CB^j\subset \CV^{(j)}$, then these form a filtration, $\CB^j\subset\CB^{j+1}$, $0\leq j<k-1$.
\end{lem}

\begin{proof}
Suppose there is an element $B\in\CB^j$ such that $B\notin\CB^{j+1}$ for some $j<k-1$. Since 
$B\in\CV^{(j)}\subset\CV^{(j+1)}$, it can be used to complete
$\CB^{j+1}$ to a basis for $\CV^{(j+1)}$. Let vector field $X$ extend a basis for $\CB^j$ to a basis for 
$\CV^{(j)}$ and note that $B\wedge X\neq 0$. 
If $\{B_1, B_2,\ldots,B_{r_{j+1}}\}$ is a basis for $\CB^{j+1}$ then 
since $X\in\CV^{(j)}\subset\CV^{(j+1)}$, there are functions $a^i$ and $b$ on $M$ such that
$X=a^iB_i+b B$ and therefore
$$
[B, X]=[B, a^iB_i]+[B, b B]\equiv a^i[B, B_i]\mod\CV^{(j+1)}.
$$
Now $[B, B_i]\notin\CV^{(j+1)}$ else $\CV^{(j+1)}$ would be integrable and the derived length of $\CV$ would be less than $k$. On the other hand $[B,X]\in\CV^{(j+1)}$ since $B, X\in\CV^{(j)}$. This implies that 
$a^i[B, B_i]=0$. Now there are linearly independent vector fields $Z_1,\ldots, Z_{r_{j+1}}$ such that
$[B, B_i]\in\mathrm{span}\{Z_i\}$ by Theorem \ref{BryantIffWeber} and this implies that $a_i=0$, for
$1\leq i\leq r_{j+1}$. This in turn implies that $X=bB$ and therefore $X\in\CB^j$, contradicting the hypothesis that $X$ completes $\CB^j$ to a basis for $\CV^{(j)}$. This can only be resolved if, in fact, $B\in\CB^{j+1}$.
\end{proof}

\begin{thm}\label{GoursatIFFbryant}
Let $\CV\subset TM$ be a bracket generating sub-bundle on manifold $M$ defining a control system of derived length $k>1$. Then $\CV$ is locally diffeomorphic to a Brunovsk\'y normal form if and only if each derived bundle $\CV^{(i)}$,$0\leq i<k$ contains an integrable Bryant sub-bundle.
\end{thm}

\begin{proof}
Suppose each derived bundle $\CV^{(i)}$ contains an integrable Bryant sub-bundle $\CB_i$, $0\leq i<k$. 
Let $\t$ be a regular first integral of $\CB^{k-1}$; that is, \newline $d\t(\CB^{k-1})=0$, $d\t\neq 0$. Let $I^{(i)}$ denote the annihilator of $\CV^{(i)}$ for each $0\leq i\leq k$. By Lemma \ref{BryantBundlesTelescope}, the 
$\CB^i$ form a filtration
$\CB^0\subset\CB^1\subset\cdots\subset\CB^{k-1}$ and therefore $d\t(\CB^i)=0$, 
$I^{(i)}(\CB^i)=0$ and $d\t\notin I^{(i)}$, for $0\leq i\leq k-1$. Since $\CB^i$ has corank 1 in $\CV^{(i)}$, we have 
$$
\CB^i=\ker\left(d\t\oplus I^{(i)}\right), \ \ \ \ \ 0\leq i\leq k-1.
$$
By Proposition \ref{bryant2}, $\chi^j:=\text{rank}\,\ch\CV^{(j)}=2m_j-m_{j+1}-1$, 
$m_j=\text{rank}\,\CV^{(j)}$, for $0\leq j\leq k-1$. There are two cases to consider: $\Delta^2_{j+1}=0$
and $\Delta^2_{j+1}\neq 0$. In the former case
$$
\chi^j=2m_j-m_{j+1}-1=m_j-1-\Delta_{j+1}=m_j-1-\Delta_j=m_{j-1}-1,
$$
to which we shall return. In section \ref{WeberAndBryantSect} we showed that if each element $\CV^{(i)}$ in the derived flag of $\CV$ contains a Bryant sub-bundle then it has the form
$$
\CV^{(j-1)}=\mathrm{span}\{T, \xi_1,\ldots,\xi_{\chi^{j-1}}, B_1^{j-1},\ldots,B^{j-1}_{\Delta_j}\}
$$
where $T\t=1$ and $\CB^{j-1}=\mathrm{span}\{\xi_1,\ldots,\xi_{\chi^{j-1}}, B_1^{j-1},\ldots,B^{j-1}_{\Delta_j}\}$
and each $B^{j-1}_l$ has degree 1 in $\CV^{(j-1)}$. That is, 
$$
[T, B_l]\equiv \mathrm{span} \{Z_l\}\mod\CV^{(j-1)},\ \ \ Z_l\in TM/\CV^{(j-1)}, \ \ 1\leq l\leq \Delta_j,
$$
where $Z_1\w Z_2\w\cdots\w Z_{\Delta_j}\neq 0$. Similarly, since $\CB_j\subset\CV^{(j)}$, we have
$$
\begin{aligned}
\CV^{(j)}=&\mathrm{span}\{T, \xi_1,\ldots,\xi_{\chi^{j-1}}, B_1^{j},\ldots,B^{j}_{\Delta_{j+1}}\}\cr
=&\mathrm{span}\{T, \xi_1,\ldots,\xi_{\chi^{j-1}}, B_1^{j-1},\ldots,B^{j-1}_{\Delta_j}\}
 \oplus\mathrm{span}\{[T, B_1^{j-1}], [T, B_2^{j-1}],\ldots, [T, B_{\Delta_j}^{j-1}]\}\cr
=&\CV^{(j-1)}\oplus\mathrm{span}\{[T, B_1^{j-1}], [T, B_2^{j-1}],\ldots, [T, B_{\Delta_j}^{j-1}]\}.
\end{aligned}
$$
Since $T\t\neq 0$ and both $\CB_{j-1}$ and $\CB_j$ are integrable and have corank 1 in $\CV^{(j-1)}$ and 
$\CV^{(j)}$, respectively, we deduce that
$$
\CB^{j-1}=\mathrm{span}\{\xi_1,\ldots,\xi_{\chi^{j-1}}, B_1^{j-1},\ldots,B^{j-1}_{\Delta_j}\}\subset\ch\CV^{(j)}
$$
and hence
$$
\ch\CV^{(j)}_{j-1}:=\CV^{(j-1)}\cap\ch\CV^{(j)}=\CB^{j-1}.
$$
That is, $\ch\CV^{(j)}_{j-1}$ is integrable and in the case $\Delta^2_{j+1}=0$, we deduce that
$\ch\CV^{(j)}_{j-1}=\ch\CV^{(j)}$. In either case, $\Delta^2_{j+1}=0$
and $\Delta^2_{j+1}\neq 0$, we have shown that $\CV$ has the refined derived type of a Goursat bundle and its intersection bundles are all integrable and agree with the Bryant sub-bundles.

Furthermore, if $\Delta_k>1$, we saw in section \ref{WeberAndBryantSect} that $\CB^{k-1}$ agrees with the resolvent bundle $R\big(\CV^{(k-1)}\big)$ which is therefore integrable and has first integral $\t$. We have therefore shown that $\CV$ is a Goursat bundle with independence condition $d\t$. By Theorem \ref{GGNF}, 
$\CV$ is locally diffeomorphic to a partial prolongation of the contact distribution on $J^1(\B R, \B R^m)$ 
(i.e., a Brunovsky normal form), where $m$ is the number of controls. 

Conversely, let $\CV$ is a Goursat bundle with first integral $\t$ of $R\big(\CV^{(k-1)}\big)$ if $\Delta_k>1$
or of $\ch\CV^{(k-1)}$ if $\Delta_k=1$. Then each intersection bundle $\ch\CV^{(j)}_{j-1}$ is integrable and has corank 1 in $\CV^{(j-1)}$, $1\leq j\leq k-1$. If $\Delta_k>1$ then the resolvent bundle exists and is integrable. By the uniqueness of Bryant bundles $\CB^{j-1}=\ch\CV^{(j)}_{j-1}$, $1\leq j\leq k-1$ are the integrable Bryant sub-bundles. Furthermore, the Bryant sub-bundle $\CB^{k-1}$ is given by the resolvent bundle
in the case $\Delta_k>1$ or by the highest order bundle $\Pi^k$ if $\Delta_k=1$.
\end{proof}

A theorem equivalent to Theorem \ref{GoursatIFFbryant} for Goursat bundles with signatures of the special  form
$\k=\langle 0,\ldots,0,\rho_k\rangle$, $\rho_k>0$ is given in \cite{RespondekContact}.

\vskip 3 pt
We can now apply Theorem \ref{GoursatIFFbryant} to prove that  if a control system is \sfl\ then its quotient by its control admissible symmetries will be a \sfl\ control system.

\begin{thm}\label{GoursatQuotientsRGoursat}
Let $(M, \CV, G)$ be a \sfl\ control system on manifold $M$ and $G$ a Lie subgroup of the Lie pseudogroup of automorphisms of $\CV$ acting by control admissible symmetries. Then the quotient control system 
$(M/G, \CV/G)$ is also \sfl.
\end{thm}
\begin{proof}
Since $\CV$ is \sfl\ it is a Goursat bundle with independence condition $dt$ and hence by Theorem \ref{GoursatIFFbryant}, each derived bundle $\CV^{(i)}$ contains an integrable Bryant sub-bundle. Then in the augmented distribution $\wh{\CV}:=\CV\oplus\bsy{\G}$, each Cauchy bundle contains $\bsy{\G}$ and each Bryant sub-bundle contains the Cauchy bundle by Proposition \ref{bryant1}. Furthermore, $G$ preserves any Bryant sub-bundle 
$\CB_j\subset\CV^{(j)}$. For, 
$$
\CV^{(j)}=\mathrm{span}\{T, \xi_1,\ldots,\xi_{\chi^j}, B_1^j,\ldots, B_{\Delta_j}^j\}=\mathrm{span}\{T\}\oplus\CB^j
$$
and $[\bsy{\G}, \CB^j]\equiv \l T\mod\CB^j$. But since $Tt=1$ and $[\bsy{\G}, \CB^j]t=0$ we have $\l=0$.
Therefore $\wh{\CV}^{(j)}$ in the derived flag of $\wh{\CV}$
contains an integrable Bryant sub-bundle for each $j$ such that $\text{rank}\wh\CV^{(j)}\leq \dim M$. Thus, by Theorem
\ref{GoursatIFFbryant}, $\wh{\CV}$ is a Goursat bundle and it follows that it will have independence condition 
$dt$ because $t$ is a first integral of $\bsy{\G}$. Hence $\CV/G$ will be \sfl.
\end{proof}

In Theorem \ref{ell-transverseQuotient-refDerTypeTHM} we made no assumptions about the distribution $\CV$
other than being $G$-invariant and bracket generating. That theorem then gives the precise relation between the refined derived type of the distribution $\CV$ and that of its quotient $\CV/G$. By Theorem \ref{GGNF}, the refined derived type of a distribution is only a partial characterizing invariant of a Brunovsk\'y normal form
(see Example \ref{example1}) and for our purposes it is important that we ultimately characterize the quotients of
feedback linearizable control systems. We settled this question in Theorem \ref{GoursatQuotientsRGoursat}. But in that theorem we did not explain how the refined derived type of the quotient is related to given invariant control system and the action of the symmetry group. This is the purpose of Theorem \ref{RDTofQuotient}.

\begin{thm}\label{RDTofQuotient}
Let $\CV$ be a bracket generating distribution over an $n$-dimensional manifold $M$ of derived length $k>1$ that is invariant under the action of an $r$-dimensional Lie group $G$ with sub-bundle $\bsy{\G}$ of infinitesimal generators.
Suppose the action is $\ell$-{\it transverse for} $\CV$ for some $1\leq\ell\leq k-1$. If $\CV$ is a Goursat bundle then the quotient 
$\CV/G$ has the refined derived type of a Goursat bundle with type numbers given by $\wh{m}_i-r$,  
$\wh{\chi}^i-r,$ and $\wh{\chi}^j_j-1$
for $0\leq i\leq \wh{k}-1$ and $1\leq j \leq \wh{k}-1$ where  $\wh{\cdot}
$  refers to quantities associated with the augmented bundle $\wh{\CV}:=\CV\oplus\bsy{\G}$. Furthermore,
each transverse bundle 
$\bsy{\G}_{j-1}$ is integrable for $1\leq j\leq \wh{k}$. 

\vskip 10 pt

\end{thm}
\begin{proof}
Since $\CV$ is a Goursat bundle then $\chi^j_{j-1}=m_{j-1}-1$ and by Theorem 
\ref{ell-transverseQuotient-refDerTypeTHM}, 
\begin{equation}\label{intersectionBundleEquality}
\wh{\chi}^j_{j-1}={\chi}^j_{j-1}+r-p'_j+q'_j=m_{j-1}-1+r-p'_j+q'_j.
\end{equation}
In general, we have by Proposition \ref{maxIntersections} that $\wh{\chi}^j_{j-1}\leq \wh{m}_{j-1}-1$ for
$1\leq j\leq \wh{k}-1$, and with $j<k$, $\CV^{(j)}$ is not integrable. Thus
$
m_{j-1}+r-p'_j+q'_j\leq\wh{m}_{j-1}
$.
However, by Theorem \ref{ell-transverseQuotient-refDerTypeTHM}, $\wh{m}_{j-1}=m_{j-1}+r-r_{j-1}$ so that 
$m_{j-1}+r-p'_j+q'_j\leq m_{j-1}+r-r_{j-1}$;
therefore, $r_{j-1}-p'_j+q'_j\leq 0$. But $q'_j\geq 0$ and by the definition of $p'_j$ and $r_{j-1}$, we deduce that $r_{j-1}\geq p'_j$ and thus, $r_{j-1}-p'_j\leq -q'_j$. This implies $q'_j\leq 0$ and therefore $q'_j=0$ and
$r_{j-1}=p'_j$. Substituting into \eqref{intersectionBundleEquality} we easily arrive at
$\wh{\chi}^j_{j-1}=\wh{m}_{j-1}-1$, which is the rank of the intersection bundle of a Goursat bundle, as we wanted to show.

To establish the relevant formula $\wh{\chi}^i=2\wh{m}_i-\wh{m}_{i+1}-1$, we note that each term in the derived flag of $\wh{\CV}=\CV\oplus\bsy{\G}$ contains an integrable Bryant sub-bundle by Theorem \ref{GoursatIFFbryant} and hence the formula
follows from Proposition \ref{bryant2}.

At this stage, we have proven that the augmented bundle $\wh{\CV}$ has the refined derived type numbers of a relative Goursat bundle. In order to pass to the refined derived type numbers of $\CV/G$ we apply Theorem 4.2 of \cite{DTVa} to $\wh{\CV}$. That is, in the reading of Theorem 4.2 in \cite{DTVa}, replace $\CV$ with $\wh{\CV}$ and replace $\bar{\CV}$ with $\CV/G$. Hence $\CV/G$ has the refined derived type numbers of a Goursat bundle. 

To prove that the transverse bundles $\bsy{\G}_i$ are integrable, we know that 
$\ch\wh{\CV}^{(j)}_{j-1}$ is integrable for each $1\leq j\leq\wh{k}-1$ and since $p'_j=r_{j-1}$,
$\bsy{\G}_{j-1}\subset\ch\CV^{(j-1)}\subset\ch\CV^{(j)}$ and hence 
$\bsy{\G}_{j-1}\subset\ch\CV^{(j)}_{j-1}$. Since $\ch\CV^{(j)}_{j-1}$ is integrable, we have
$$
[\bsy{\G}_{j-1}, \bsy{\G}_{j-1}]=\bsy{\G}_{j-1}^{(1)}\subset\ch\CV^{(j)}_{j-1}\subset\CV^{(j-1)}
$$
and as $\bsy{\G}$ is integrable, we have $\bsy{\G}_{j-1}^{(1)}\subset\bsy{\G}$. We deduce that
$\bsy{\G}_{j-1}^{(1)}\subset\bsy{\G}\cap\CV^{(j-1)}=\bsy{\G}_{j-1}$, which concludes the proof.
\end{proof}

\begin{thm}
Let $\CV$ be a Goursat bundle of derived length $k$ and signature $\k=\langle \r_1, \r_2, \ldots, \r_k\rangle$
on manifold $M$ of dimension $P$. Furthermore, let $\bsy{\G}$ be the bundle generated by the infinitesimal generators of the smooth, transitive, regular action of any local Lie group of symmetries $G$ of dimension
$r<P-\sum_{i=1}^k\r_i$, and assume that $\bsy{\G}$ is strongly transverse to $\CV$. Moreover, the signature $\bar{\kappa}$ of $\CV/G$ is given by
$$
\bar{\kappa}=\left\langle \nabla^2_2-\Delta^2_2,\ 
\nabla^2_3-\Delta^2_3,\ldots, \nabla^2_{\wh{k}}-\Delta^2_{\wh{k}}, 
\Delta_{\wh{k}}-\nabla_{\wh{k}}\right\rangle.
$$
\end{thm}

\vskip 5 pt
\begin{proof}
This follows from Theorem \ref{GoursatQuotientsRGoursat} and the relationship between $\wh{\CV}$
and $\CV/G$ given by Theorem \ref{relGoursatThm}. For $\wh{k}<k$, $\bsy{\G}+\ch\CV^{(\,\wh{k}\,)}_{\wh{k}-1}$ is an integrable Bryant sub-bundle by Theorem \ref{GoursatQuotientsRGoursat}. In case $\wh{k}=k$ then the same reasoning applies by uniqueness of Bryant subbundles. The only difference is that we have either $\Pi^{\wh{k}}+\bsy{\Gamma}$ or 
$R\left(\CV^{(\,\wh{k}-1\,)}\right)+\bsy{\Gamma}$ as the Bryant subbundle when $\rho_k=1$ or $\rho_k>1$, respectively. Since each Bryant sub-bundle uniquely defines an integrable Weber structure, the resolvent bundle of $\wh{\CV}:=\CV\oplus\bsy{\G}$ is integrable and we can conclude that $\wh{\CV}$ is a relative Goursat bundle and therefore $\CV/G$ is a Goursat bundle. Note that the dimension restriction on $G$ is necessary to make the quotient well-defined. The signature of $\CV/G$ in the theorem statement then follows from Corollary \ref{decelVhat}.
\end{proof}

\section{Generalized S-G-S Test for Linearizability}\label{GSStestSection}
The G-S algorithm \cite{GS92} is a procedure for determining a \stf\ transformation that transforms a fully nonlinear control system \eqref{controlSystem} to some Brunovsk\'y normal form. There is an associated extremely simple test \cite{GS92, sluisPhD} for checking the possible {\it existence} of such a transformation for a given autonomous control system. However the test works in exactly the same way for time-varying control systems; see \cite{DaSilvaGS}.  It is expressed in terms of the Pfaffian system 
$I$ defining the control system as in \eqref{controlPfaffian}. Note that we denote this Pfaffian system by $I$ instead of $\bsy{\o}$ as in \eqref{controlPfaffian}. The notion of derived system is dual to the case of vector fields. The derived bundle $I^{(1)}$ of $I$ is defined by
$$
I^{(1)}=\{\o\in I~|~d\o\equiv 0\mod I\};
$$ 
and we iterate as before: $I^{(j)}=\left(I^{(j-1)}\right)^{(1)}$.

Let $(M, I)$ be the Pfaffian system on manifold $M$ that represents the control system and let $t$ denote the time variable as in \eqref{controlPfaffian}. Let
$I^{(j)}$ denote the $j^{\text{th}}$ element in the derived flag \eqref{DerFlagI} of $I$, where the derived length of $I$ is 
$k>1$,
\begin{equation}\label{DerFlagI}
\{0\}=I^{(k)}\subset I^{(k-1)}\subset\cdots\subset I^{(1)}\subset I.
\end{equation}
\begin{thm}[G-S-S test, \cite{GS92}, \cite{sluisPhD}]
Let  $I$ be the Pfaffian system representation \eqref{controlPfaffian} of any smooth control system \eqref{controlSystem}. Then
\eqref{controlSystem} is static feedback linearizable if and only if 
$$
dI^{(j)}\equiv 0\!\!\!\!\!\!\mod \big(I^{(j)}\oplus\mathrm{span}\{dt\}\big),\ \ \ 0\leq j\leq k.
$$
\end{thm}
Based on our previous discussion and the generalized Goursat normal form, in this section we give a proof of this and at the same time we generalize it to give a test for the existence of {\it some} local diffeomorphism $\vf: J^\k\to M$ (not necessarily a \stf\ transformation) such that $\vf^*I$ is the contact sub-bundle 
$\bsy{\b}^\k\subset TJ^\k$ on the jet space $J^\k$.  This test can be applied to {\it any} Pfaffian system, not only to control systems. But when applied to a control system it provides a simple test for the orbital feedback linearizability of any smooth control system when a time scale $\t$ is given.

We shall first require the following result.

\begin{thm}\label{generalizedGStest}
Let
$(M, I)$ be the Pfaffian system on manifold $M$ of derived length $k>1$ with
derived flag \eqref{DerFlagI}. Suppose $\CV^{(k-1)}:=\ker I^{(k-1)}$ has an integrable Bryant sub-bundle with regular first integral $\t$ on $M$.
%$$
%\{0\}=I^{(k)}\subset I^{(k-1)}\subset\cdots\subset I^{(1)}\subset I.
%$$
%Suppose 
%\begin{enumerate}
%\item If $\Delta_k>1$ then
%$\CV^{(k-1)}:=\ker I^{(k-1)}$ has an integrable Weber structure and $\t$ is one of its regular first integrals. 
%\item If $\Delta_k=1$, let $\t$ be a regular first integral of $\ch\CV^{(k-1)}$.
%\end{enumerate}
Then $I$ defines a Goursat bundle with independence condition $d\t\neq 0$ for its integral submanifolds, if and only if
$$
I^{(j)}\oplus\mathrm{span}\{d\t\}
$$
is integrable for $0\leq j\leq k-1$. 
\end{thm}

\begin{proof}
Let $\CB^{k-1}\subset\CV^{(k-1)}$ be the integrable Bryant sub-bundle satisfying $d\t(\CB^{k-1})=0$, 
and suppose $I^{(i)}\oplus\mathrm{span}\{d\t\}$ is integrable, $0\leq i\leq k-1$. Then 
$\CB^i=\ker\,\left(I^{(i)}\oplus\mathrm{span}\{d\t\}\right)$
is a corank 1 integrable distribution in $\CV^{(i)}$ and therefore a Bryant sub-bundle in 
$\CV^{(i)}$. By Theorem \ref{GoursatIFFbryant}, $\CV$ is a Goursat bundle with independence condition
$d\t\neq 0$.
Conversely, if $\CV$ is a Goursat bundle with independence condition $d\t$ then by Theorem \ref{GoursatIFFbryant}, each
Bryant sub-bundle $\CB^i\subset \CV^{(i)}$, $0\leq i\leq k-1$ exists and is integrable. Hence, 
$\mathrm{ann}\CB^i:=d\t\oplus I^{(i)}$ is integrable for each such $i$.
\end{proof}
\noindent Note that if $\rho_k>1$ then $\CB^{k-1}$ agrees with the integrable resolvent of $\CV^{(k-1)}$ and if $\rho_k=1$ it agrees with the highest order bundle $\Pi^k$. 

\begin{cor}[Generalized S-G-S test]\label{generalizedGStest1}
Let
$(M, I)$ be the Pfaffian system on manifold $M$ of derived length $k>1$ such that $I^{(k)}=\{0\}$. Suppose there is a smooth, real-valued function $\t$ on $M$, with $d\t\neq 0$ such that
$I^{(j)}+\mathrm{span}\{d\t\}$ is integrable for each $j\geq 0$. Then $I$ is diffeomorphic to a Brunovsky normal form with independence condition $d\t$.
\end{cor}
\begin{proof}
We have $d\t\notin I^{(j)}$, $0\leq j<k$, else we get a contradiction of the triviality of  $I^{(k)}$. Hence
$\ker \left(I^{(j)}\oplus\mathrm{span}\{d\t\}\right)$ is a Bryant sub-bundle and the result follows by Theorem \ref{generalizedGStest}.
\end{proof}

\begin{cor}[Sluis-Gardner-Shadwick test]
If $(M, I)$ is the Pfaffian system \eqref{controlPfaffian} of any smooth control system 
\eqref{controlSystem} evolving with time $\t=t$, then $I$ will be static feedback linearizable if and only if 
$$
I^{(j)}\oplus\mathrm{span}\{dt\}
$$
is integrable for $0\leq j\leq k-1$. 
\end{cor}
\begin{proof}
Let the Pfaffian system representation of the control system be the $
I$ in Corollary \ref{generalizedGStest1} and take $\t=t$.
\end{proof}
Thus, Theorem \ref{generalizedGStest} generalizes the S-G-S test so that it applies to any Pfaffian system with independence condition $d\t$. In particular, it gives a simple test for the orbital feedback linearizability of any smooth control system \eqref{controlSystem}, relative to a given time scale $\t$. The determination of time scales will be described explicitly in future work. 

We conclude with simple test for \sfl\ quotient control systems analogous to the S-G-S test.

\begin{thm}[S-G-S test for quotients]
Let $(M, \CV)$ be a control system and $G$ its Lie group of control admissible

symmetries with sub-bundle of infinitesimal generators $\bsy{\G}$; and suppose 
$\wh{\CV}:=\CV\oplus\bsy{\G}\subset TM$ has derived length $\wh{k}>1$. Then the quotient $\CV/G$ will be \sfl\ if and only if
$
\wh{I}^{\,\, (j)}\oplus\mathrm{span}\{dt\}
$
is integrable for $0\leq j\leq \wh{k}$, where $\wh{I}=\mathrm{ann}\,\wh{\CV}$.
\end{thm}

\begin{proof}
This follows from Theorems \ref{linearizableQuotientExGoursat} and \ref{generalizedGStest}.
\end{proof}
%%%%%%%%%%%%%%%%%%%%%%%%%%%%%%%%%%%%%%%%%%%%%%%%%%%%%%%%%%%%%%%%%%%%%%%%%%

%%%%%%%%%%%%%%%%%%%%%%%%%%%%%%%%%%%%%%%%%%%%%%%%%%%%%%%%%%%%%%%%%%%%%%%%%%%%%%%%%%
\section{Extended Example: PVTOL}\label{PVTOLsection}
The following control system from \cite{Hauser} is a classic model in the nonlinear control theory literature and has been a primary example of many dynamic feedback linearization methods. The PVTOL control system is given by 
\begin{equation}\label{PVTOL}
\begin{aligned}
\ddot{x}&=-u_1\,\sin(\theta)+h\,u_2\,\cos(\theta), \\
\ddot{z}&=u_1\,\cos(\theta)+h\,u_2\,\sin(\theta)-g, \\
\ddot{\theta}&=\lambda u_2,
\end{aligned}
\end{equation} 
where $u_1$ and $u_2$ are the controls, $x,\dot{x},z,\dot{z},\theta,\dot{\theta}$ are states and $h,g,$ and $\lambda$ are constants. We can normalize $g$ and $\lambda$ to one. 
\subsection{Refined Derived Type}
Let $\mcal{V}$ be the dsitribution representing the PVTOL system. Then the derived flag of $\mcal{V}$ is given by
\begin{equation}\label{pvtol df}
\begin{aligned}
\der{V}{1}&=\mcal{V}\oplus\mathrm{span}\left\{\sin(\theta)\partial_{x_1}-\cos(\theta)\partial_{z_1},-h\cos(\theta)\partial_{x_1}-h\sin(\theta)\partial_{z_1}-\partial_{\theta_1}\right\},\\
\der{V}{2}&=\der{V}{1}\oplus\mathrm{span}\left\{-\sin(\theta)\partial_x+\theta_1\cos(\theta)\partial_{x_1}+\cos(\theta)\partial_z+\theta_1\sin(\theta)\partial_{z_1},\phantom{+s}\right.\\
\,&\phantom{\der{V}{1}\oplus}\left. h\cos(\theta)\partial_x+h\theta_1\sin(\theta)\partial_{x_1}+h\sin(\theta)\partial_z-h\theta_1\cos(\theta)\partial_{z_1}\right\},\\
\der{V}{3}&=TM.
\end{aligned}
\end{equation}
The Cauchy bundles of $\mcal{V}$ and the intersection bundles happen to agree and are given by 
\begin{equation}\label{pvtol filt}
\begin{aligned}
&\Chare{V}\quad=&\hspace{-3cm}\span\{0\},\\
&\Char{V}{1}=&\inCharOne{V}=\mathrm{span}\left\{\partial_{u_1},\partial_{u_2}\right\},\\
&\Char{V}{2}=&\Char{V}{2}_1=\mathrm{span}\left\{\partial_{u_1},\partial_{u_2}\right\}.
\end{aligned}
\end{equation}
It follows that the refined derived type of $\mcal{V}$ is given by 
\begin{equation}\label{pvtol rdt}
\mathfrak{d}_r\left(\mcal{V}\right)=[[3,0],[5,2,2],[7,2,2],[9,9]],
\end{equation}
and thus \eqref{PVTOL} is not SFL or even OFL. 
\subsection{Control Admissible Symmetries}
To find a dynamic feedback linearization of an invariant control system 

cite{CKV24}, we first find its control symmetries.
\begin{thm}
The PVTOL control system has an eight dimensional control  symmetry algebra isomorphic to a Lie algebra with Levi decomposition
\begin{equation}
\frak{sl}(2,\mathbb{R})\oplus \frak{s}_{5,9},
\end{equation}
where $\frak{s}_{5,9}$ is a 5-dimensional solvable Lie algebra described in \cite{SnoblWinternitz}.
\end{thm}
\begin{proof} 
By use of commands in the \textbf{contact} and \textbf{derive11} procedures in MAPLE we find that the Lie algebra of infinitesimal generators of this Lie group are spanned by the following vector fields on $M$:
\begin{equation}\begin{aligned}
X_1 =&\,h\sin^2(\theta)\cos(\theta)\partial_x+ h\theta_1(3\cos^2(\theta)-1)\sin(\theta)\partial_{x_1}+(x-h\sin(\theta)\cos^2(\theta))\partial_z\\
&+(x_1+2h\theta_1\cos(\theta)-3h\theta_1\cos^3(\theta))\partial_{z_1}+\sin^2(\theta)\partial_\theta+\theta_1\sin(2\theta)\partial_{\theta_1}\\
&+\cos(\theta)\sin(\theta)(5h\theta_1^2-u_1)\partial_{u_1}+(2\theta_1^2\cos(2\theta)+u_2\sin(2\theta))\partial_{u_2}\\
X_2=&\,h\sin(\theta)\cos^2(\theta)\partial_x+h\theta_1\cos(\theta)(3\cos^2(\theta)-2)\partial_{x_1}-(h\cos^3(\theta)+t^2/2+z)\partial_z\\
&+(3h\theta_1\cos^2(\theta)\sin(\theta)-z_1-t)\partial_{z_1}+\cos(\theta)\sin(\theta)\partial_\theta+\theta_1\cos(2\theta)\partial_{\theta_1}\\
&+((5h\theta_1^2-u_1)\cos^2(\theta)-2h\theta_1^2)\partial_{u_1}+(u_2\cos(2\theta)-2\theta_1^2\sin(2\theta))\partial_{u_2}\\
X_3=&(t^2/2+z)\partial_x+(t+z_1)\partial_{x_1}-x\partial_z-x_1\partial_{z_1}-\partial_\theta\\
X_4=&(x-h\sin(\theta))\partial_x+(x_1-h\theta_1\cos(\theta))\partial_{x_1}+(t^2/2+h\cos(\theta)+z)\partial_z\\
&+(t+z_1-h\theta_1\sin(\theta))\partial_{z_1}+(u_1-h\theta_1^2)\partial_{u_1}\\
X_5=&t\partial_x+\partial_{x_1}\\
X_6=&\partial_x\\
X_7=&t\partial_z+\partial_{z_1}\\
X_8=&\partial_z
\end{aligned}\end{equation}
As an abstract Lie algebra $\mathfrak{g}\cong \bsy{\G}_\text{alg}$, if we make the $\mathbb{R}$-linear change of basis 
$$X_2\mapsto-2X_2-X_4,\quad X_3\mapsto X_1+X_3$$
then the multiplication table for the associated Lie bracket displayed below helps us identify the Lie algebra of control symmetries of the PVTOL control system. 
\begin{figure}[h]
\[\begin{array}{|l|l|l|l|l|l|l|l|l|}
\hline
 & X_1 & X_2 & X_3 & X_4 & X_5 & X_6 & X_7 & X_8 \\ \hline
X_1 & \phantom{-}0  & \phantom{-}2X_1 & -X_2 & \phantom{-}0 & -X_7 & -X_8 & \phantom{-}0 & 0 \\ \hline
X_2 & -2X_1 & \phantom{-}0  & \phantom{-}2X_3 & \phantom{-}0 & \phantom{-}X_5 & \phantom{-}X_6 & -X_7 & -X_8 \\ \hline
X_3 & \phantom{-}X_2 & -2X_3 & \phantom{-}0 & \phantom{-}0 & \phantom{-}0 & \phantom{-}0 & -X_5 & -X_6 \\ \hline
X_4 & \phantom{-}0  & \phantom{-}0  & \phantom{-}0 & \phantom{-}0 & -X_5 & -X_6 & -X_7 & -X_8 \\ \hline
X_5 & \phantom{-}X_7 & -X_5 & \phantom{-}0 & \phantom{-}X_5 & \phantom{-}0 & \phantom{-}0 & \phantom{-}0 & \phantom{-}0 \\ \hline
X_6 & \phantom{-}X_8 & -X_6 & \phantom{-}0 & \phantom{-}X_6 & \phantom{-}0 & \phantom{-}0 & \phantom{-}0 & \phantom{-}0 \\ \hline
X_7 & \phantom{-}0  & \phantom{-}X_7 & \phantom{-}X_5 & \phantom{-}X_7 & \phantom{-}0 & \phantom{-}0 & \phantom{-}0 & \phantom{-}0 \\ \hline
X_8 & \phantom{-}0 & \phantom{-}X_8 & \phantom{-}X_6 & \phantom{-}X_8 & \phantom{-}0 & \phantom{-}0 & \phantom{-}0 & \phantom{-}0 \\ \hline
\end{array}\]
\caption{Multiplication table of the infinitesimal control symmetries of the PVTOL.}
\end{figure}
\newline
From here we see that there is a Levi decomposition $\mathfrak{g}=\frak{h}\oplus\mathfrak{r}$ where $\mathfrak{h}$ is the semisimple Lie algebra $\mathfrak{sl}(2,\mathbb{R})$ generated by $\{X_1,X_2,X_3\}$ and $\frak{r}=\mathrm{span}\{X_4,X_5,X_6,X_7,X_8\}$ is the radical of $\mathfrak{g}$. Thus, once we understand the isomorphism class of $\frak{r}$ we will have completely described the control symmetry algebra of the PVTOL system. The subalgebra $\frak{r}$ has as derived series
\[
\frak{r}^{(3)}=0\subset\frak{r}^{(2)}=\mathrm{span}\{X_5,X_6,X_7,X_8\}\subset \frak{r}
\] 
Let $R$ be the simply connected Lie group for $\frak{r}$ and let $(x_1,\cdots,x_5)$ be the coordinates on $\frak{r}^*$. Then the infinitesimal coadjoint action of $R$ on $\frak{r}^*$ is given by 
\[
\Gamma_{\frak{r}^*}=\mathrm{span}\{x_1\partial_{x_1}+x_2\partial_{x_2}+
x_3\partial_{x_3}+x_4\partial_{x_4},\, \partial_{x_5}\}
\]
The invariant functions of the infinitesimal action are generated by
\begin{equation}
\text{Inv}_{\Gamma_{\frak{r}^*}}=\mathrm{span}\left\{\frac{x_1}{x_2},\frac{x_1}{x_3},\frac{x_1}{x_4}\right\}. 
\end{equation}
Elements of the set $\text{Inv}_{\Gamma_{\frak{r}^*}}$ above are called \textit{Casimir invariants} of 
$\frak{r}$. The dimension, solvability, and Casimir invariants of $\frak{r}$ are enough to uniquely identify this Lie algebra up to Lie algebra isomorphism. Using this information, we can conclude that $\frak{r}$ is isomorphic to the real Lie algebra $\frak{s}_{5,9}$ described in Chapter 18 of \cite{SnoblWinternitz}. Therefore, the full control symmetry Lie algebra of the PVTOL system is isomorphic to 
\begin{equation}
\frak{sl}(2,\mathbb{R})\oplus \frak{s}_{5,9}.
\end{equation}
\end{proof}
\begin{prop}
Let $\CV$ be the distribution representing the PVTOL system on manifold $M$. Then $M$ is locally a Lie group with Lie algebra $\mathfrak{g}$ which has Levi decomposition
\[
\mathfrak{sl}(2,\mathbb{R})\oplus \mathfrak{s}_{6,140}
\]
where $\mathfrak{s}_{6,140}$ is a six dimensional solvable Lie algebra with nilradical $\mathfrak{n}_{5,1}$ described in \cite{SnoblWinternitz}.
\end{prop}
\begin{proof}
The Lie algebra $\mathfrak{s}_{5,9}$ has the property that $[\partial_t,X_4]=-X_7,\,[\partial_t,X_5]=X_8,$ and 
$[\partial_t,X_7]=X_6$ while commuting with all other elements $X_i$'s. Let $X_9=-\partial_t$. Then $\mathfrak{g}=\mathrm{span}_\mathbb{R}\{X_1,\ldots,X_9\}$ and from the previous theorem it is clear that 
$\mathfrak{g}$ will have the proposed Levi decomposition for some solvable 6-dimensional Lie algebra 
$\mathfrak{s}=\mathrm{span}_{\mathbb{R}}\{X_4,\ldots,X_9\}$. Let $S$ be the connected and simply connected Lie group such that $T_eS\cong \mathfrak{s}$. Using the coadjoint action of $S$ on 
$\mathfrak{s}^*$ we can calculate the associated Casimirs and we find that the Casimirs are generated by 
$\{x_5/x_3,(x_4x_5-x_2x_3)/(x_3x_5)\}$. Thus, using the table in Chapter 19 of \cite{SnoblWinternitz}, we recover that 
$\mathfrak{s}=\mathfrak{s}_{6,140}$. 
Moreover, since the infinitesimal generators span $TM$ pointwise, we have that $M$ is locally any Lie group realizing $\mathfrak{g}$ as its Lie algebra. If $M$ is connected and simply connected then this Lie group is unique. Therefore, the PVTOL control system may be viewed as an invariant distribution on a 9-dimensional Lie group. 
\end{proof}
\subsection{Static Feedback Linearizable Quotients}
The PVTOL admits many SFL quotient systems from symmetry. Interestingly, we can never obtain an SFL quotient system from a 1-parameter subgroup of the control  symmetry group of the PVTOL. 
\begin{prop}
Let $\mcal{V}$ be the distribution representing the PVTOL control system on manifold $M$ and $G$ its control  admissible symmetry group and $H$ any 1-parameter subgroup of $G$. Then $\mathcal{V}/H$ is never a Goursat bundle on $M/H$. 
\end{prop}
\begin{proof}
In order that $\mcal{V}/H$ represent a control system, it is necessary that the sub-bundle of infinitesimal generators $\Gamma_H=\mathrm{span}\{ X_H\}$ be strongly transverse to $\mcal{V}$ (i.e. at least 1-transverse). This means that either $\Gamma_H$ is 1-transverse or 2-transverse. If $\Gamma_H$ is 2-transverse, then by Theorem \ref{ell-transverseQuotient-refDerTypeTHM} the augmented bundle $\wh{\mcal{V}}$ has refined derived type 
\[
[[4,1],[6,3,3],[8,3,3],[9,9]].
\]
As, $\wh{\chi}^2_1=3$, $\wh{\CV}$  cannot be a relative Goursat bundle since for this it is necessary that 
$\wh{\chi}^2_1=5$. Thus no 1-parameter control admissible symmetry group of $\mcal{V}$ with 2-transverse action yields a SFL quotient system. 
Now assume that $\Gamma_H$ is strictly 1-transverse. Then the refined derived type of the augmented bundle is
\[
[[4,1],[6,3,3],[7,3+q_2',3+q_2],[9,9]]
\]
where $q_2'$ and $q_2$ are as in Theorem \ref{ell-transverseQuotient-refDerTypeTHM}. Note there is no contribution from $p_2$ and $p_2'$ because $\Gamma_H$ is 1-transverse and hence $\Gamma_H\cap \Char{V}{1}=\{0\}$ which implies $\Gamma_H\cap \Char{V}{2}=
\{0\}$ since $\Char{V}{2}=\Char{V}{1}$ for the PVTOL system. Thus $p_2=p_2'=0$, though this is not a necessary detail for the proof. A necessary condition that a bundle be Goursat is that $\Delta_i\geq \Delta_{i+1}$ for all $1\leq i\leq k$ where $k$ is the derived length. But here we notice that $\wh{\Delta}_3>\wh{\Delta}_2$ and thus $\wh{\mcal{V}}$ is not a relative Goursat bundle, hence the relevant quotient system is not a Goursat bundle. 
\end{proof}

On the other hand, we have that all 2-dimensional and 2-transverse control admissible symmetry groups give rise to SFL quotient systems. 
\begin{prop}\label{linearizableQuotientsProp}
Let $\mcal{V}$ be the distribution representing the PVTOL control system on manifold $M$ and $G$ its control admissible symmetry group and $H$ any 2-dimensional subgroup of $G$ with 2-transverse action on $M$. Then $\mathcal{V}/H$ is always a Goursat bundle on $M/H$ via a static feedback transformation. 
\end{prop}
\begin{proof}
Since $H$ acts 2-transversely and has dimension $2=\Delta_3$ then we may apply Corollary \ref{tot-transverse-ranks-cor} and we find that the refined derived type of the relative Goursat bundle is 
\[
[[5,2],[7,4,4],[9,9]]
\]
which is easily seen as the refined derived type of a relative Goursat bundle. Moreover, we know that 
\[
\Char{\wh{V}}{1}=\inCharOne{\wh{V}}=\Char{V}{1}\oplus \Gamma_H=\inCharOne{V}\oplus\Gamma_H.
\]
which are clearly integrable and we can easily check that 
\[
\mcal{B}=\inCharOne{\wh{V}}\oplus\mathrm{span}_{C^\infty(M)}\{\sin(\theta)\partial_{x_1}-\cos(\theta)\partial_{z_1},-h\cos(\theta)\partial_{x_1}-h\sin(\theta)\partial_{z_1}-\partial_{\theta_1}\}
\]
is a corank-1 integrable sub-bundle of $\der{V}{1}$ and therefore is an integrable Bryant sub-bundle. Hence by Theorem \ref{BryantIffWeber} and Theorem \ref{linearizableQuotientExGoursat} we can conclude that $\wh{\mcal{V}}$ is a relative Goursat bundle. 
\end{proof}
The two previous Propositions  give an indication of the type of results that can be deduced from Theorem \ref{ell-transverseQuotient-refDerTypeTHM} regarding the existence and no-existence of \sfl\ quotients. If one is interested in finding SFL quotients of a particular signature, or looking for maximal linearizable subsystems induced by symmetry, then one can use the rank conditions to explore  Lie subalgebras of interest.
Proposition \ref{linearizableQuotientsProp} shows how quotients by low-dimensional Lie groups lead to 
SFL quotients by an analysis of the transversality of the group action. 
%%%%%%%%%%%%%%%%%%%%%%%%%%%%%%%%%%%%%%%%%%%%%%%%%%%%%%%%%%%%%%%%%%%%%%%%%%%%%%%%%%
\section{An Open Question} 

The larger the symmetry group, the more possibilities that there exist many SFL quotients. This suggests that invariant control systems on certain Lie groups may have the property that they always admit SFL quotients. Indeed, we observed this in Section \ref{PVTOLsection}. What Lie groups have this property? Is this property purely dimension dependent? With this in mind, an interesting follow-up question is: how large must $G$ be such that $\CV/G$ have the refined derived type of a Goursat bundle? Moreover, given the work in \cite{CKV24}, what can we deduce about dynamic feedback linearization of control systems with many symmetries?

%---------------------------------------------------------------
%
% BibTeX users please use
%%\bibliographystyle{alpha}
%%\bibliography{CtrlRef-GQ}
%
% Non-BibTeX users please follow the syntax
% the syntax of "referenc.tex" for your own citations
%\input{referenc}
%%%%%%%%%%%%%%%%%%%%%%%%%%%%%%%%%%%%%%%%%%%%%%%%%%%%%%%%%%%%%%%%%%%%%%

%

%\begin{thebibliography}{[LTY99b]}

\bibliographystyle{alpha}
\bibliography{CtrlRef-GQ}

%\end{thebibliography}
%%%%%%%%%%%%%%%%%%%%%%%%%%%%%%%%%%%%%%%%%%%%%%%%%%%%%%%%%%%%%%%%%%%%%%
%
\end{document}